\documentclass[lettersize,journal]{IEEEtran}
\usepackage{defs}
\usepackage{amsmath,amssymb,amsthm,amsfonts}
\usepackage{microtype}
\usepackage[utf8]{inputenc}
\usepackage[T1]{fontenc}
\usepackage{xcolor}
\usepackage{cite}
\usepackage{amsmath,amssymb,amsfonts}
\usepackage{algorithmic}
\usepackage{graphicx}
\usepackage{textcomp}
\usepackage{orcidlink}
\usepackage{mathtools}
\usepackage{subcaption}
\usepackage{url}
\usepackage{todonotes}
\usepackage[linesnumbered,ruled,lined,algo2e]{algorithm2e} 
\usepackage{derivative}
\usepackage{bbm}
\usepackage{titlesec}

\usepackage{tabularray}   
\usepackage{adjustbox}   

\usepackage{mathtools}

\newtagform{vardiamondtag}{(}{)\,\,\,\(\vardiamond\)}

\titlespacing{\section}{0pt}{5pt plus 1pt minus 1pt}{2pt plus 1pt minus 1pt}
\titlespacing{\subsection}{0pt}{4pt plus 1pt minus 1pt}{1.5pt plus 1pt minus 1pt}

\usepackage[font=small,skip=2pt]{caption}
\definecolor{metablue}{HTML}{0064E0}
\definecolor{metafg}{HTML}{1C2B33}
\definecolor{metabg}{HTML}{F1F4F7}
\definecolor{metabgdeep}{HTML}{D9EFFF}
\definecolor{metagreen}{HTML}{EAFFE8}
\definecolor{metagreen}{HTML}{FCFFEE}
\definecolor{metared}{HTML}{FFEAE8}

\definecolor{darkbrown}{HTML}{4C0202}
\definecolor{brown}{HTML}{76210E}
\newcommand{\highlight}[1]{{\color{metablue} \textbf{#1}}}

\usepackage[framemethod=tikz,xcolor=true]{mdframed}

\newmdenv[backgroundcolor=metabgdeep, roundcorner=10pt, skipabove=4pt, linewidth=0pt, innertopmargin=4pt]{myframe}

\newmdenv[backgroundcolor=metabgdeep, roundcorner=2pt, skipabove=4pt, linewidth=0pt, innertopmargin=4pt]{myOCP}

\newmdenv[backgroundcolor=metared, roundcorner=10pt, skipabove=7pt, linewidth=0pt, innertopmargin=7pt]{myalgo}
\newmdenv[%
    leftmargin=0.5cm,
    backgroundcolor=yellow!10,%
    roundcorner=5pt,%
    tikzsetting={draw=red, line width=2.0pt}%
    ]{SpecialText}%

\newtheorem{theorem}{\textbf{Theorem}}[section]
\newtheorem{lemma}{\textbf{Lemma}}[section]

\newtheorem{assumption}{\textbf{Assumption}}[section]
\newtheorem{remark}{\textbf{Remark}}[section]
   
\renewcommand{\IEEEQED}{\IEEEQEDopen}
\makeatletter
\def\theorem@qed{\pushQED{\qed}\qedhere\popQED}
\makeatother

\def\BibTeX{{\rm B\kern-.05em{\sc i\kern-.025em b}\kern-.08em
    T\kern-.1667em\lower.7ex\hbox{E}\kern-.125emX}}
\usepackage{enumitem}
\setlist[enumerate, 1]{topsep=-0.5ex, parsep=0.5ex}
\setlist[itemize, 1]{topsep=-0.5ex, parsep=0.5ex}

\DeclareMathOperator{\PCont}{PC}
\DeclareMathOperator{\Cont}{C}

\DeclareMathOperator{\sbjto}{s.\hspace{-0.5mm}t.}

\DeclareMathOperator{\Ortho}{O}

\renewcommand{\le}{\ensuremath{\leqslant}}
\renewcommand{\ge}{\ensuremath{\geqslant}}
\renewcommand{\leq}{\le}

\let\emptyset\varnothing

\DeclareSymbolFont{extraup}{U}{zavm}{m}{n}
\DeclareMathSymbol{\varheart}{\mathalpha}{extraup}{86}
\DeclareMathSymbol{\vardiamond}{\mathalpha}{extraup}{87}
\DeclareMathSymbol{\varclub}{\mathalpha}{extraup}{84}
\DeclareMathSymbol{\vardspade}{\mathalpha}{extraup}{85}
 \DeclarePairedDelimiterX\lm[2]\lparen\rparen{#1;#2}
 \DeclarePairedDelimiterX\expecof[1]\lbrack\rbrack{#1}
\DeclarePairedDelimiterX\cexpecof[1]\lbrack\rbrack{\def\given{\: \delimsize\vert\:}#1}

\DeclarePairedDelimiterX\cprobof[1]\lparen\rparen{\def\given{\: \delimsize\vert\:}#1}

\makeatletter
\renewcommand*\env@matrix[1][*\c@MaxMatrixCols c]{%
  \hskip -\arraycolsep
  \let\@ifnextchar\new@ifnextchar
  \array{#1}}
\makeatother

\DeclarePairedDelimiterX\aset[1]\lbrace\rbrace{\def\suchthat{\; \delimsize\vert\;}#1}
 \DeclarePairedDelimiterX\lcrc[2]\lbrack\rbrack{#1,#2}
\DeclarePairedDelimiterX\lcro[2]\lbrack\lbrack{#1,#2}
 \DeclarePairedDelimiterX\lorc[2]\rbrack\rbrack{#1,#2}
\DeclarePairedDelimiterX\loro[2]\rbrack\lbrack{#1,#2}
\let\oldsqrt\sqrt
\def\sqrt{\mathpalette\DHLhksqrt}
\def\DHLhksqrt#1#2{%
\setbox0=\hbox{$#1\oldsqrt{#2\,}$}\dimen0=\ht0
\advance\dimen0-0.2\ht0
\setbox2=\hbox{\vrule height\ht0 depth -\dimen0}%
{\box0\lower0.4pt\box2}}

\makeatletter
\newcommand{\firstpageheadline}[1]{%
  \gdef\@firstpageheadline{#1}%
}
\gdef\@firstpageheadline{}

\let\old@logohead\@logohead
\def\@logohead{%
\bgroup
\footnotesize
\savebox{\@tempboxa}{\vtop {%
\hbox to \textwidth{\hfill \@firstpageheadline \hfill}%
\hbox to \textwidth{}%
\hbox to \textwidth{}%
\hbox to \textwidth{}%
\hbox to \textwidth{}%
}}%
\box\@tempboxa
\egroup
}
\makeatother

\begin{document}


\title{\(\algoname\): Optimal Transport-Driven Attack Detection in Cyber-Physical Systems
}
\author{Souvik Das \orcidlink{0000-0001-6918-6219} and Siddhartha Ganguly \orcidlink{0000-0003-2046-2061}
\thanks{S. Das is with the Department of Information Physics and Computing, The University of Tokyo, Japan. S. Ganguly is with the Daniel Guggenheim School of Aerospace Engineering, Georgia Institute of Technology, USA.  Emails: \{\textsf{souvikd@g.ecc.u-tokyo.ac.jp, sganguly41@gatech.edu}\}.}
}


\maketitle
\thispagestyle{empty}
\begin{abstract}
This letter presents an optimal-transport (OT)-driven, distributionally robust attack detection algorithm, \(\algoname\), for cyber–physical systems (CPS) modeled as partially observed linear stochastic systems. The underlying detection problem is formulated as a minmax optimization problem using 1-Wasserstein ambiguity sets constructed from observer residuals under both the nominal (attack-free) and attacked regimes, and show that the minmax detection problem can be reduced to a finite-dimensional linear program for computing the worst-case distribution (WCD). Off-support residuals are handled via a kernel-smoothed score function that drives a CUSUM procedure for sequential detection. We also establish a non-asymptotic tail bound on the false-positive error of the CUSUM statistic under the nominal (attack-free) condition, under mild assumptions. Numerical illustrations are provided to evaluate the robustness properties of \(\algoname\). 
\end{abstract}


\section{Introduction, motivation, and background}\label{sec:intro}
Cyber-physical systems (CPSs), such as power grids, transportation networks, and manufacturing infrastructures, are large-scale interconnected systems that tightly integrate computation, communication, and control for continuous monitoring and operation. Due to their scale and limited supervision, they remain vulnerable to faults, anomalies, and adversarial attacks. Residual-driven fault detection is a well-known signal-processing framework \cite{ref:AnnRev:Con:fang2007fault,ref:GuoZhaWanFan:IEEE:SigPross:Residue} for monitoring and securing such critical infrastructures. In this approach, an observer or filter is applied to the measured outputs, and the resulting residual sequence is analyzed to detect deviations from the nominal operating regime toward anomalous behavior.

Classical detectors, including likelihood-ratio and CUSUM-type schemes (including robust versions) \cite{ref:huber1965robust,ref:AN-AT-AA-SD-2023,ref:BGCA:lu2022genetic,ref:cusum_new_1,ref:cusum_new_2,ref:cusum_new_3,ref:cusum_new_4}, are effective when the nominal and anomalous residual distributions are accurately specified. Their performance, however, can deteriorate under distributional and non-Gaussian disturbances \cite{ref:SD-PD-DC-23}.  In contrast to the majority of literature that studies attack detection problems under a control-theoretic framework, this letter adopts an optimal transport (OT) and distributionally robust viewpoint for attack detection in  CPSs.

OT has emerged as a principled tool for comparing empirical distributions and quantifying signal discrepancies under model mismatch, finite-sample uncertainty, and non-Gaussian variability \cite{ref:OT:in:SP:and:ML:Review,ref:OT:in:SP:and:ML:Wass:Fourier}. These features make it extremely suitable for detection problems, where the nominal and anomalous residual laws are rarely known exactly. Instead of assuming exact residual distributions, we construct Wasserstein ambiguity sets around empirical residual distributions obtained from nominal and attacked data. The detection problem is then posed as a min-max binary hypothesis test over these ambiguity sets. This viewpoint leads to a \emph{robust signal classification rule} whose worst-case residual distributions can be computed from data.

\textbf{Related works:} A few works employ distributionally robust optimization for attack detection. Specifically, \cite{ref:YF:DL:CS:WassDetection:FDI} uses OT-driven Wasserstein ambiguity sets for the unknown disturbance distribution and optimizes a reachability-based performance metric, to design a parity-space detector. \cite{ref:VR:NH:JR:TS:WassAnomaly} instead work with moment-based ambiguity sets and then compute ellipsoidal attack-reachable sets; this gives worst-case false alarm rate control under second-moment information, but the detector itself remains a conventional quadratic test, with no Wasserstein geometry between benign/attack distributions and no explicit min–max optimal test or sequential statistic. For sensor attacks, \cite{ref:DL:SM:LCSS:WassDetection} leverages Wasserstein metrics, but as a distance between a benchmark residual distribution and a sliding-window empirical distribution, recomputed online via a linear program; sequential detection and false positive bounds were not studied.
In contrast, our main contributions are:
\begin{itemize}[leftmargin=*,label = \(\circ\)]
\item \textbf{An OT-driven detection:} For a stochastic linear time-invariant (LTI) plant equipped with a steady-state observer, we devise a new data-driven residual-based detector, leveraging tools from OT, DRO, and hypothesis testing. We formulate the detection problem as a min-max hypothesis test between two ambiguity sets of probability measures. Drawing results from \cite{ref:WassHypTest} and Kantorovich–Rubinstein duality \cite{ref:CV:OTbook}, we show that the computation of worst-case distribution (WCD) reduces to a \emph{finite linear program (LP)}; informally
\begin{myOCP}
\highlight{Informal Theorem A:} Let \(Q_1,Q_2\) be the empirical measures of residuals on the finite set \(\widehat\metspaceX \Let \aset[]{s_{\ell}}_{\ell=1}^n\) (training window), and let
\begin{align}
\mathcal{P}_k \Let \aset[\big]{P  \suchthat \Wdist_1(P,Q_k)\le \eps_k}\quad \text{for }k\in\aset[]{1,2}.\nn
\end{align}
The minmax testing risk equals
\[
\eps^\star = 1 - V^{\star},
\]
where $V^\star$ is the optimum of a \emph{finite linear program}. 
\end{myOCP}
See Theorem \ref{prop:finite:LP:formulation} for further information.

\item \textbf{Kernel smoothing and sequential CUSUM-based test:} While the LP exactly solves the on-support detection problem, real-time deployment must evaluate the test on \emph{new residuals} \(z\) which may not be in the training set. To extend beyond the on-support training data, we present a \emph{kernel-smoothing} construction that maps the discrete WCDs to continuous densities and uses the log-density ratio as a continuous score, attractive for sequential operation.

\item \textbf{Non-asymptotic bounds on the false-positive error:} Beyond algorithmic tractability, we provide \emph{non-asymptotic guarantees} on the \emph{tail bound} for the false positive error; informally
\begin{myOCP}
\highlight{Informal Theorem B:} Run the CUSUM recursion \(\testsignal_0 = 0\), \(\testsignal_t=\max\aset[]{0,\testsignal_{t-1}+\increment_t}\). Under mild conditions on the increments, for any threshold \(h > 0\):
\begin{align}
\PP_1(\testsignal_t \ge h) \le 2 \exp\Bigl(-h^2/8 \sum_{i=1}^t \sigma_i^2\Bigr), \nn
\end{align}  
where \(\sigma_i>0\) known.
\end{myOCP}
See Theorem \ref{thrm:cusum:azuma-type:bounds:new} for further details.

\item \textbf{Non-asymptotic bounds on the post-attack detection guarantee:} We provide non-asymptotic post-attack guarantees on the probability of detection delay, and also compute the average detection delay under distributional uncertainties. An informal version of the result is summarized below:
\begin{myOCP}
    \highlight{Informal Theorem C: }
    Let \(\PP_{\nu}\) be the probability measure corresponding to an attack initiated at time \(\nu\). Under some mild conditions, there exist constants \(\theta>0\) and \(c_{\theta}>0\) such that for every \(d\in\N\):
    \begin{align*}
    \PP_{\nu}\bigl(\tau_h-\nu\ge d\,\big|\,\tau_h\ge\nu\bigr)\le
    \min\left\{1,\,\exp\bigl(\theta h-dc_{\theta}\bigr)
    \right\},
    \end{align*}
    Moreover, the conditional average detection delay satisfies
    \begin{align*}
    \EE_{\nu}\bigl[\tau_h-\nu\,\big|\,\tau_h\ge\nu\bigr]\le\left\lceil
    \frac{\theta h}{c_{\theta}}\right\rceil+\frac{1}{\exp(c_{\theta})-1}.
    \end{align*}
\end{myOCP}
We refer the readers to Theorem \ref{thm:post:attack:delay} for further details.

\end{itemize}
\vspace{1mm}
\noindent The \embf{key features} of this work are as follows: \textbf{(a)} Our detection scheme \embf{does not} require full distributional knowledge of the process noise, only that it lies in an ambiguity set specified by 1-Wasserstein distance, ensuring distributional robustness.  \textbf{(b)} Our detection scheme is \emph{adversary-agnostic}, with no assumptions on the attack model nor on the policy employed by the adversary. \textbf{(c)} We assume the adversary has access to system parameters, nominal control policies, and sensor measurements. They can learn steady-state behavior, adapt to system changes, collude, and launch attacks accordingly to disrupt attack-free performance.

\section{Preliminaries and problem description}\label{sec:prob:form}

Let \(d \in \N\) be a natural number and \(\Omega \subset \Rbb^d\). We denote the standard probability simplex by \(\simplex{d} \Let \aset{\zeta \in \Rbb^d_{\succeq 0} \mid \sum_{i=1}^d \zeta_i = 1}\). Fix natural numbers \(n,m \ge 1\), and let \(\sourcevec \Let (\sourcevec_1,\ldots,\sourcevec_m)^{\top} \in \simplex{m}\) and \(\tarvec\Let (\tarvec_1,\ldots,\tarvec_n)^{\top} \in \simplex{n}\). On \(\metspaceX\) define the discrete measures \(\measX \Let \sum_{i=1}^m \sourcevec_i \delta_{x_i}\) and \(\measY \Let \sum_{j=1}^{n} \tarvec_j \delta_{y_j}\), where \(x_i,y_j \in \metspaceX\) are the support points. Define the set of couplings (which is a convex polytope) by 
\begin{align}
\probmeas(\sourcevec,\tarvec) \Let \bigg\{\coupmat  \in \Rbb^{m \times n}_{+} \bigg \lvert \sum_{j=1}^n\coupmat_{ij} = \sourcevec_i, \sum_{i=1}^m \coupmat_{ij} = \tarvec_j \bigg\}.\nn
\end{align}
The Kantorovich-Rubinstein distance \cite[Chapter 6]{ref:CV:OTbook} takes the form (for \(p=1\))
\begin{align}\label{eq:disc_W_distance}
\Wdist_1(\measX, \measY) \Let  \inf_{\coupmat \in \probmeas(\sourcevec,\tarvec)} \sum_{i=1}^{m}\sum_{j=1}^n \coupmat_{ij} \costX(x_i,y_j).
\end{align}
In the sequel, will construct our ambiguity sets employing \eqref{eq:disc_W_distance} with the ground cost \((z,z') \mapsto \costX(z,z') \Let \norm{z-z'}_2\).

\subsection{Stochastic linear systems}
\label{subsec:stochastic linear systems}
Let \(\dimst,\dimcon,\dimdist\) and \(\dimout\) be natural numbers. Consider a time-invariant discrete-time control system
\begin{align}
    \label{eq:system}
    \st_{t+1} = A \st_t + B \cont_t + E \dist_t,\,\,
    y_t = C \st_t + F \dist_t,
\end{align}
with \(\st_0 \Let \xz\) given and \(t \in \Nz\), along with the following data: \(\st_{t} \in \Rbb^{\dimst}\), \(\cont_t \in \Rbb^{\dimcon}\), \(\dist_t \in \Rbb^{\dimdist}\), and \(y_t \in \Rbb^{\dimout}\) are the vectors representing the states, control inputs, uncertainties, and the output at time \(t\), with \(A \in \Rbb^{\dimst \times \dimst}\), \(B \in \Rbb^{\dimst \times \dimcon}\), \(E \in \Rbb^{\dimst \times \dimdist}\), \(C \in \Rbb^{\dimout \times \dimst}\) and \(F \in \Rbb^{\dimout \times \dimdist}\).  The controller computes the state estimate, which is a function of the observation \((y_t)_{t \in \Nz}\), to estimate and monitor the system, and employ control actions that are a function of the state estimate. Let \(L \in \Rbb^{\dimst \times \dimout}\) be a steady-state gain, we consider the estimator:
\begin{align}
    \label{eq:observer}
    \obst_{t+1} &= A \obst_t + B \cont_t + L(y_t - C \obst_t),\,
    \residual_t  \Let y_t - C \obst_t,
\end{align}
with \(\obst_0 \Let \widehat{\st}\). Moreover, \(L\in \Rbb^{\dimst \times \dimout}\) is the steady-state gain picked such that \((A-LC)\) is Schur stable. 
\begin{assumption}
    \label{assum:standing assump}
    \((A, B, C)\) is both stabilizable and detectable. 
\end{assumption}
We model the influence of the adversary by the stochastic process \(\advdis\) that enters the system through the output channel: for \(t \in \Nz\), we have the recursion
\begin{align}
\label{eq:adv sys}
    \adst_{t+1} = A \adst_t + B \adcont_t + E \dist_t,\,\adobs_t = C \adst_t + F \dist_t + \advdis_t,
\end{align}
where \(\adst_0 \Let \xz\) is given. Here, the controller's action \(\adcont_t\)  at each time is based on the available corrupted output \((\adobs_t)_{t \in \Nz}\). The corresponding state estimator is given by the recursion
\begin{align}
    \label{eq:observer in adv}
    &\adobst_{t+1} = A \adobst_t + B \adcont_t + L(\adobs_t - C \adobst_t) \text{ with }\adobst_0 \Let \widehat{\st}^{\mathrm{a}}, \nn\\
    &\adres_t  \Let \adobs_t - C \adobst_t = C (\adst_t - \adobst_t)  + F \dist_t + \advdis_t,
\end{align}
where \(L \) is chosen similarly.
Under the above setting, we seek to infer whether the system's uncertainties are inherent, represented by \(\dist\), or they have been influenced by an external adversarial disturbance denoted by \(\advdis\). Informally, we pose this problem as \embf{a composite hypothesis testing problem}: given a sample \(\sample\) of the residual obtained from \eqref{eq:observer} and \eqref{eq:observer in adv}, we set: 
\begin{myOCP}
\begin{align*}
    \text{(H0): }\sample \sim P_1 \in \mathcal{P}_1 \text{ and }
    \text{(H1): }\sample \sim  P_2 \in \mathcal{P}_2,
\end{align*}
\end{myOCP}
where \(\mathcal{P}_1\) and \(\mathcal{P}_2\) are the sets of probability measures defined over \(\Omega\); see \eqref{eq:wass:ambiguity:set} for a definition.

\section{\(\algoname:\) algorithm and properties}\label{sec:main_result}
We are ready to formulate the detection problem. Recall that the underlying distributions of \(\dist\) and \(\advdis\) are not available. However, we have access to \(n_1 \in \N\) and \(n_2 \in \N\) number of residual samples as the \emph{training data}. Let \(\metspaceX \Let \Rbb^{\dimout}\). Consider the data streams generated from \eqref{eq:observer} and \eqref{eq:observer in adv}, respectively: \(\metspaceX_1 \Let \aset[\big]{\resdis_i \in \metspaceX}_{i=1}^{n_1}\) and \(\metspaceX_2 \Let \aset[\big]{\resadv_j \in \metspaceX}_{j=1}^{n_2},\)
and let \(\widehat{\metspaceX} \Let \metspaceX_1 \sqcup \metspaceX_2 = \aset[]{s_{\ell}}_{\ell=1}^{n}\) with \(n\Let n_1+n_2\). Note that both \(\metspaceX_1\) and \(\metspaceX_2\) are collected by online implementation of \eqref{eq:system}-\eqref{eq:observer} and \eqref{eq:adv sys}-\eqref{eq:observer in adv}. Let \(\alpha \in \simplex{n_1},\beta \in \simplex{n_2}\); we construct the empirical measures 
\begin{align}\label{eq:empirical measures}
    Q_1 \Let \sum_{i=1}^{n_1} \alpha_i \delta_{\resdis_i} \quad \text{and} \quad Q_2 \Let \sum_{j=1}^{n_2} \beta_j \delta_{\resadv_j},
\end{align}
with \(\alpha_i = \frac{1}{n_1}\) and \(\beta_j = \frac{1}{n_2}\) for all \((i,j) \in \aset[]{1,\ldots,n_1} \times \aset[]{1,\ldots,n_2}\). 
We also define the index sets \(\indexset_1 \Let \aset[]{1,\ldots,n_1}\) and \(\indexset_2 \Let \aset[]{n_1+1,\ldots,n}\) with \(\widehat{\indexset} \Let \aset{1, \ldots,n}\). In shorthand notation, for \(\ell=1,\ldots,n_1\) we have \(s_{\ell} = \resdis_{\ell}\) and for \(\ell=n_1+1, \ldots,n\), we have \(s_{\ell} = \resadv_{\ell-n_1}\). 
We expand \eqref{eq:empirical measures} in the unified index \(\ell\), and towards this end, we define for \(\ell \in \aset[]{1,\ldots,n}\), the empirical measures \((Q_1)_{\ell} \Let \alpha_{i}\) for \(\ell \in \indexset_1\), and \((Q_2)_{\ell} \Let \beta_{\ell-n_1}\) for \(\ell \in \indexset_2\) under the nominal and attacked case, respectively. 
Note that for \(k=1,2\), \((Q_k)_{\ell} = 0\) for \(\ell \in \widehat{\indexset}\setminus \indexset_k \) because their corresponding supports are disjoint.

For some \(\eps_1,\eps_2>0\), we consider two ambiguity sets centred around \(Q_k\) for \(k =1,2\), given by
\begin{align}\label{eq:wass:ambiguity:set}
\hspace{-2mm}\polish_k \Let \aset[\big]{P_k \in \polish(\metspaceX) \suchthat \Wdist_1(P_k,Q_k) \le \eps_k} \text{ for }k \in \{1,2\},
\end{align}
where \(\Wdist_1(\cdot,\cdot)\) is the \(1\)-Wasserstein distance defined in \eqref{eq:disc_W_distance}. 

\begin{assumption}\label{assum:blanket}
Throughout this paper, we assume that the ambiguity radii \((\eps_1, \eps_2)\) are selected so that the ambiguity sets \eqref{eq:wass:ambiguity:set} are disjoint.
\end{assumption}
This condition is required to exclude the degenerate case in which the same probability law is admissible under both \((\text{H0})\) and \((\text{H1})\). Mathematically, it means that the radii should satisfy \(\eps_1+\eps_2 < \Wdist_1(Q_1,Q_2).\)\footnote{We thank one of the anonymous reviewers for this pointer.}

\subsection{The minmax test} We formulate a minmax test based on the ideas advanced in \cite{ref:WassHypTest,ref:WassHypTest:NIPS:}. Recall that, given hypotheses (H0) and (H1), a \emph{randomized test} is any Borel measurable map \(\test:\metspaceX \lra \lcrc{0}{1}\), which for any observation \(s_{\ell}\in \metspaceX\), accepts the hypothesis (H0) with probability equal to \(\test(s_{\ell})\) and (H1) with probability equal to \(1 - \test(s_{\ell})\).\footnote{Throughout, we abuse notation by referring to random vectors and their realizations using the same symbols.} For a pair \((P_1,P_2) \in \polish_1(\metspaceX) \times \polish_2(\metspaceX)\), define 
\begin{align}\label{eq:risk}
    \risk \bigl(\test;P_1,P_2\bigr) \Let \EE_{1}\expecof[]{1-\test(s_{\ell})} + \EE_{2}\expecof[]{\test(s_{\ell})},
\end{align}
where (expectations) \(\EE_1\) and \(\EE_2\) are defined as per \(\PP_1\) and \(\PP_2\), respectively. Then, the robust testing problem is 
\begin{equation}
	\label{eq:RobTest}
	\begin{aligned}
		& \inf_{\test} \sup_{P_1 \in \polish_1, P_2 \in \polish_2} &&  \risk(\test;P_1,P_2).
	\end{aligned}
\end{equation}
Note that the risk function \eqref{eq:risk} accounts for the trade-off between likelihood of incurring \emph{Type-}\(1\) and \emph{Type-}\(2\) errors by the test \(\test(\cdot)\) under \(P_1\) and \(P_2\), respectively, thereby reducing the problem into a \embf{signal classification problem.}

\subsection{LP formulation}
In \cite{ref:WassHypTest,ref:WassHypTest:NIPS:} it was shown that in the setting of \emph{simple hypothesis testing}, the optimal \emph{on sample} test takes a similar form as the likelihood ratio test. More precisely: 

\begin{myOCP}
\begin{lemma}
\label{lemm:Risk:Expression}
For any fixed \((P_1,P_2) \in \polish_1 \times \polish_2\),
consider the inner minimization problem \(\optrisk(P_1,P_2) \Let \inf_{\test:\metspaceX \lra \lcrc{0}{1}} \risk(\test;P_1,P_2)\) corresponding to the minmax problem \eqref{eq:RobTest}. Then the Neyman-Pearson-like randomized rule: \(\test^{\star}(s_{\ell}) = 1\) if \(\tfrac{\odif{P_2}}{\odif{(P_1+P_2)}}(s_{\ell}) > \tfrac{1}{2}\), \(\test^{\star}(s_{\ell}) = 0\) if \(\tfrac{\odif{P_1}}{\odif{(P_1+P_2)}}(s_{\ell}) < \tfrac{1}{2}\), and \(\test^{\star}(s_{\ell}) = p \in \lcrc{0}{1}\) otherwise, is optimal for \eqref{eq:RobTest} with the risk \(\optrisk(P_1, P_2) = 1 - \mathsf{TV}(P_1,P_2)\).\footnote{Where \(\mathsf{TV}\) is the total variation distance \cite[Chapter 1]{ref:CV:OTbook}.} \hfill \(\vardiamond\)
\end{lemma}
\end{myOCP}

To find the worst-case distributions (WCD), we look at the inner supremum problem in \eqref{eq:RobTest}, i.e., 
\begin{align}\label{eq:RobTest:SupProb}
\sup_{P_1 \in \mathcal{P}_1, P_2 \in \mathcal{P}_2} 1 - \mathsf{TV}(P_1,P_2).
\end{align}
The variational problem \eqref{eq:RobTest:SupProb} is infinite-dimensional and is not computationally viable in general. However, due to the nature of the \(\Wdist_1\)-ambiguity set and the discreteness of the empirical measures, we can optimally solve \eqref{eq:RobTest:SupProb} via a finite-dimensional tractable convex program, which is the focus of the next result. 

Note that \eqref{eq:RobTest:SupProb}, if it admits a solution, ensures that the overlapping distributions \(P_1\) and \(P_2\) have to be close in the total variation sense (\(TV\)-sense), thereby constraining the adversarial policy. In other words, an adversary cannot arbitrarily degrade the performance of the CPS, or it risks detection.

\begin{myOCP}
\begin{theorem}\label{prop:finite:LP:formulation}
Recall that the notations established in \S\ref{sec:prob:form}. Corresponding to the ground cost \(\costX(\cdot,\cdot)\), we define the pairwise cost matrix \(D\in\Rbb_{+}^{n\times n}\) with \(D_{\ell m}\Let \costX(s_{\ell},s_m)\). Consider the robust testing problem 
\begin{align}
\eps{^\star} \Let \inf_{\test}\ \sup_{P_1\in \polish_1,\ P_2\in \polish_2} \EE_{1}\expecof[]{1-\test(s_{\ell})} + \EE_{2}\expecof[]{\test(s_{\ell})},
\end{align}
where the ambiguity sets are defined in \eqref{eq:wass:ambiguity:set}. 
\begin{enumerate}[label=\textup{(\alph*)}, leftmargin=*, widest=b, align=left]

\item Then \(\eps^{\star} = 1 - V^{\star}\). Where \(V^{\star}\) is the optimal value of the following \emph{finite linear program} with decision variables \(\widehat{p}\Let (p_1,p_2,\Gamma_1,\Gamma_2,t)\): \(p_1,p_2\in\Rbb^{n}_{+}\) with \(\sum_{\ell}p_{k,\ell}=1\), \(\Gamma_1,\Gamma_2\in\Rbb_{+}^{n \times n}\) (optimal transport couplings from empirical measures \(Q_k\) to \(p_k\)), and \(t\in\Rbb^{n}_{+}\): 
\begin{equation}\label{eq:LFD:LinProg} 
\begin{aligned}
& \hspace{-5mm}\max_{\widehat{p}}	&& \sum_{\ell=1}^{n} t_\ell \\
& \hspace{-5mm} \sbjto		&& \hspace{-6mm} \begin{cases}
\sum_{\ell=1}^{n}\sum_{m=1}^{n} \Gamma_{1,\ell m}\, D_{\ell m}\ \le\ \eps_1,\\ \sum_{\ell=1}^{n}\sum_{m=1}^{n} \Gamma_{2,\ell m}\, D_{\ell m}\ \le\ \eps_2, \\
\sum_{m=1}^{n} \Gamma_{k,\ell m} = (Q_k)_{\ell},\,
\sum_{\ell=1}^{n} \Gamma_{k,\ell m} = p_{k,m},\\
0 \le t_{\ell} \le p_{1,\ell},\,\, 0 \le t_\ell \le p_{2,\ell},\,p_{k,\ell}\ge 0\\
\sum_{\ell=1}^{n} p_{k,\ell}=1,\,
\Gamma_{k,\ell m}\ge 0 \text{ for all } k\in \aset[]{1,2}.
\end{cases}
\end{aligned}
\end{equation}
\item The LP \eqref{eq:LFD:LinProg} is well-posed and admits a solution. \hfill \(\vardiamond\)
\end{enumerate}
\end{theorem}
\end{myOCP}
\begin{proof}
For any fixed \((P_1,P_2)\), via Lemma \ref{lemm:Risk:Expression} we have 
\begin{align}
    \inf_{\test} \EE_{1}\expecof[]{1-\test(s_{\ell})} + \EE_{2}\expecof[]{\test(s_{\ell})} = 1-\mathsf{TV}(P_1,P_2), \nn
\end{align}
which implies \(\varepsilon^\star=1-\inf_{P_1\in\mathcal P_1,P_2\in\mathcal P_2}\mathsf{TV}(P_1,P_2).\)

Because \(Q_k\) is discrete on \(\widehat{\metspaceX}\), the balls \(\polish_k\) admits the standard OT-coupling representation. Indeed, since \(p_k=\sum_{m=1}^{n} p_{k,m}\,\delta_{s_m}\) on \(\widehat{\metspaceX}\), using the duality arguments and properties of \(\Wdist_1(\cdot,\cdot)\) via \cite[Lemma 6 and Lemma 8]{ref:WassHypTest}, there exists a coupling matrix \(\Gamma_k\in \Rbb_{+}^{n\times n}\) between \(Q_k\) and \(p_k\) with row sums \((Q_k)_\ell\), column sums \(p_{k,m}\), and cost bounded by \(\eps_k\), i.e., \(\sum_{m=1}^n \Gamma_{k,\ell m}=(Q_k)_\ell,\) \(\sum_{\ell=1}^n \Gamma_{k,\ell m}=p_{k,m}\) and \(\sum_{\ell,m}\Gamma_{k,\ell m}\,D_{\ell m}\le \eps_k,\) with $D_{\ell m}\Let \costX(s_{\ell},s_m)$. Note that the finite-reduction at at atoms \(s_\ell\) again follows from the Kantorovich-Rubinstein-type Duality \cite[Lemma 6]{ref:WassHypTest}.

Since both the distributions are discrete on the same finite support \(\widehat{\metspaceX}\), we have: \(\mathsf{TV}(p_1,p_2)=1-\sum_{\ell=1}^n \min\aset[]{p_{1,\ell},p_{2,\ell}}.\) Introducing variables \(t\Let (t_1,\ldots,t_n) \in\Rbb^n_{+}\) with \(0\le t_\ell\le p_{1,\ell}\), \(0\le t_\ell\le p_{2,\ell}\) with \(\ell=1,\dots,n\), we see that 
\begin{align}\label{eq:proof:LP:cost}
  &\max_{t\in \Rbb^n_{+}} \aset[\bigg]{\sum_{\ell=1}^n t_{\ell} \suchthat 0\le t_\ell\le p_{1,\ell},\, 0\le t_\ell\le p_{2,\ell}}  
\end{align}
is equal to \(\sum_{\ell=1}^n \min\aset[]{p_{1,\ell},p_{2,\ell}}\). The preceding arguments along with \eqref{eq:proof:LP:cost} yields exactly the finite LP in the proposition, hence \(\eps^{\star}=1-V^{\star}\).

We show the well-posedness and existence of solutions to \eqref{eq:LFD:LinProg}. Feasibility is immediate: choose \(
p_k = Q_k\), \(\Gamma_{k,\ell m} \Let (Q_k)_{\ell}\indic{m=\ell}\), and \(t_{\ell}\Let 0\).\footnote{Here \(\indic{A}(\cdot)\) is the standard indicator function for a given set \(A\subset \Rbb\).} Then by construction \(\sum_m \Gamma_{k,\ell m}=(Q_k)_{\ell}\) and \(\sum_{\ell} \Gamma_{k,\ell m}=p_{k,m}= (Q_k)_m\), also \(\sum_{\ell,m}\Gamma_{k,\ell m} D_{\ell m} = \sum_{\ell} (Q_k)_{\ell} D_{\ell\ell}=0\le\theta_k\) since \(D_{\ell\ell}=\costX(s_\ell,s_\ell)=0\); and \(t_\ell=0\le p_{1,\ell},p_{2,\ell}\). Hence, the feasible set is nonempty. The feasible set is defined by bounded linear equalities and inequalities in a finite-dimensional Euclidean space and thus is compact, and \((p,\Gamma,t)\mapsto \sum_{\ell} t_{\ell}\) is continuous. Existence of a solution follows immediately from the Weierstrass theorem. 
\end{proof}
\subsection{On sample test} Note that, any optimizer of \eqref{eq:LFD:LinProg} produces the WCD, with \(P_k^{\star} \Let \sum_{\ell=1}^{n} p_{k,\ell}^{\star}\delta_{s_\ell}\) for \(k \in \{1,2\}\). Then the on-support optimal test values are
\begin{align}\label{eq:on:supp:test}
\test^{\star}(s_{\ell})=\
\begin{cases}
1, & \text{if } p_{2,\ell}^{\star} > p_{1,\ell}^{\star},\\
0, & \text{if } p_{2,\ell}^{\star} < p_{1,\ell}^{\star},\\
\in\lcrc{0}{1}, & \text{if } p_{2,\ell}^{\star} = p_{1,\ell}^{\star} \, (\text{ties}).
\end{cases}
\end{align}
In case of a tie, \(\test\as(s_{\ell})\) choose either \(1\) or \(0\) w.p. \(1/2\). With this \(\test\as\), we have the worst-case risk \(\eps^{\star} = \sum_{l=1}^{n}p_{1,\ell}^{\star}(1-\test^{\star}(s_{\ell}))+ p_{2,\ell}^{\star} \test^{\star}(s_{\ell}) = 1 - \sum_{\ell=1}^{n}t_{\ell}^{\star}.\)
\begin{remark}
    \label{rem:equivalence LP}
    In general, the max-min problem \eqref{eq:RobTest:SupProb} yields an lower bound for the min-max problem originally defined in \eqref{eq:RobTest}. However, following arguments similar to those in \cite[Theorem 1]{ref:WassHypTest} based on strong duality, one can show that equality holds between the original min-max problem \eqref{eq:RobTest} and the corresponding max-min problem \eqref{eq:RobTest:SupProb}.
\end{remark}

\begin{remark}
   \label{rem:equivalence LP}
Note that, after eliminating the rows of the transport matrices associated with zero source mass, the LP \ref{eq:LFD:LinProg} contains \(n^2+3n\) decision variables. Thus, the number of decision variables, as well as the associated storage requirement, grows quadratically with the total number of training residual samples.
\end{remark}

\subsection{Kernel smoothing and a CUSUM test}\label{subsec:kernel}

The LP \eqref{eq:LFD:LinProg} identifies the WCDs
\(P_{k}^{\star}\) for \(k=1,2\), on the training atoms \(\widehat{\metspaceX}\) and the test \(\test^{\star}(s_{\ell})\) on those atoms. While this fully resolves the on-support problem, in deployment, however, incoming residuals \(z\in\metspaceX\) will rarely lie exactly on \(\widehat{\metspaceX}\). To define \(\test^{\star}\) for every \(z\), one must extend it off-support. To this end, we adopt a kernel smoothing technique, which is often preferable in practice because of its compatibility with sequential testing.

{
\renewcommand{\algorithmcfname}{\(\algoname\)}
\renewcommand{\thealgocf}{}
\begin{algorithm2e}[!ht]
\DontPrintSemicolon
\SetKwInOut{ini}{Initialize}
\SetKwInOut{giv}{Data}
\SetKwInOut{hyp}{Hyperparameter}
\ini{\(\testsignal_0 = 0\)}
\giv{Stream of offline nominal data \(\resdis\),  kernel function \(\kernel_{\sigma}\), \(h>0\)}

    Construct the training data \(\widehat{\metspaceX} \)
    
    Solve \eqref{eq:LFD:LinProg} to compute the LFDs \((P^{\star}_{1}, P^{\star}_{2})\) on \(\widehat{\metspaceX}\)

    Extend \((P^{\star}_{1}, P^{\star}_{2})\) by kernel smoothing to obtain \eqref{eq:score:function}
    
    \For{\(t = 1,2,\ldots,\)}{
    Compute the sequence \((\testsignal_t)_{t \in \Nz}\) defined in \eqref{eq:cusum:iteration}
    
    Raise an alarm if \(\testsignal_t \ge h\), else continue Step 5

    }
    
\caption{OT-based detection algorithm for CPS.}
\label{alg:seD_ord_comp}
\end{algorithm2e}
}

Choose a kernel \(\kernel:\Rbb^{d}\lra \lcro{0}{+\infty}\) with \(\int_{\Rbb^d} \kernel(t)\odif{t}=1\), \(\kernel(0) = 1\), and bandwidth \(\sigma>0\). Set \(\xi \mapsto \kernel_{\sigma}(\xi)\Let \frac{1}{\sigma^{d}}\kernel\bigl(\frac{\xi}{\sigma}\bigr)\). Then, for \(k=1,2\), define \(\metspaceX \ni z \mapsto f_k(z) \Let \sum_{\ell=1}^{n} p_{k,\ell}^{\star}\;\kernel_{\sigma}(z-s_{\ell}).\) Define the \emph{score function} by
\begin{align}\label{eq:score:function}
z \mapsto s_{\kernel}(z)& \Let \log \frac{f_2(z)}{f_1(z)} =\log\frac{\sum_{\ell=1}^n p_{2,\ell}^{\star} \kernel_{\sigma}(z-s_{\ell})}{\sum_{\ell=1}^n p_{1,\ell}^{\star} \kernel_{\sigma}(z-s_{\ell})}.
\end{align}
While there are many choices for the kernel function, we will pick the Gaussian kernel \(\xi \mapsto \kernel_{\sigma}(\xi) \Let \bigl(2\pi \sigma^2)^{-d/2}\exp{\bigl(-\norm{\xi}_2^2/2\sigma^2}\bigr)\) because of its nice properties. 
\begin{remark}[Features of the Gaussian kernel smoothing]\label{rem:nice properties}
Note that \(\xi \mapsto \kernel_{\sigma}(\xi)>0\) implies that \( z\mapsto f_k(z) >0\), this makes the score function \eqref{eq:score:function} well-defined and finite everywhere. Gaussian smoothing makes \(f_k(\cdot)\) and \(s_{\kernel}(\cdot)\) smooth and from a numerical point of view each \(\log f_k(z)\) is a log-sum-exp of quadratic terms which is stable to compute and easy to truncate to \(K\) nearest atoms for speed. Moreover \(s_{\kernel}(s_{\ell})\lra \log\frac{p_{2,\ell}^{\star}}{p_{1,\ell}^{\star}}\) as \(\sigma\downarrow 0\), which matches the Neyman-Pearson-like test on the atoms. \hfill \(\vardiamond\)
\end{remark}

Define the \emph{increments} sequence \((\increment_t)_{t \in \Nz} \Let (s_{\kernel}(z_t))_{t \in \Nz}\). We consider the CUSUM recursion \cite{ref:AN-AT-AA-SD-2023} 
\begin{align}\label{eq:cusum:iteration}
  \testsignal_0 \Let 0,\qquad \testsignal_t \Let \max\{0,\testsignal_{t-1}+\increment_t\},\quad\text{for } t\ge1,  
\end{align}
and we test \(\testsignal_t\) for every time \(t\), against some appropriately chosen threshold \(h\):
\begin{myOCP}
\begin{align*}
    \begin{cases}
        \testsignal_t \ge h \quad &\text{Reject (H0) in favor of (H1),}\\
        \testsignal_t < h & \text{Reject (H1) in favor of (H0)}.
    \end{cases}
\end{align*}    
\end{myOCP}
The method to choose the threshold is detailed in \S\ref{sec:num_exp}. We now outline all these key steps in Algorithm \(\algoname\).

\section{Theoretical guarantees}\label{sec:tg}

We establish non-asymptotic guarantees on the false positive error under (H0) with \(P_1\), incurred by the algorithm, which states that under the attack-free regime, the event \(\aset[]{\testsignal_t \ge h}\) admits a low probability. During deployment, this error serves as a practical guide to select an appropriate threshold; see \S\ref{sec:num_exp} for a discussion on how it is utilised to select a threshold.

\begin{myOCP}
\begin{theorem}\label{thrm:cusum:azuma-type:bounds:new}
Fix a finite horizon \(\horlen \in \N\). On \((\samplesp,\filtration)\) let
\(\filtration_i \Let \sigma(\increment_1,\dots,\increment_i)\), \(i=1,\dots,\horlen\) be the natural filtration.
Let \((\increment_i)_{i=1}^{\horlen}\) be an \((\filtration_i)_{i=1}^{\horlen}\)-adapted sequence with
\(\EE_1{[|\increment_i|]}<+\infty\) for every \(i=1,\dots,\horlen\).
For each \(i\), define the conditional drift and the centered increments \(\mu_i \Let \EE_1[\increment_i \mid \filtration_{i-1}]\) and \(\xi_i \Let \increment_i - \mu_i\), and assume that
under \(\PP_1\), the distribution under (H0), the conditional drift is non-positive:
\begin{align}\label{eq:assum:drift}
 \mu_i = \EE_1[\increment_i\mid \filtration_{i-1}]\ \le\ 0\quad\text{a.s. for }i=1,\dots,\horlen, \tag{A1}
\end{align}
and that there exist known constants \(\sigma_1,\ldots,\sigma_{\horlen}>0\) such that for every \(\lambda \in \Rbb\) and every \(i=1,\ldots,\horlen\):
\begin{align}\label{eq:assum:exp:bound}
\EE_1[ \exp\bigl(\lambda \xi_i \bigr) \mid \filtration_{i-1}] \le \exp\bigl( \lambda^2 \sigma_i^2/2\bigr) \quad\text{a.s.},\tag{A2}
\end{align}
where \(\EE_1\expecof[]{\cdot}\) is an expectation under \(P_1\). Define \(V_{\horlen} \Let \sum_{i=1}^{\horlen}\sigma_i^2\). Then for \(h>0\), we have 
\begingroup
\usetagform{vardiamondtag}
\begin{align}\label{eq:azuma:bound:1}
  \PP_1\Bigl(\max_{1 \le s \le \horlen} \testsignal_{s}\ge h \Bigr)
  \le 2\exp\bigl(-h^2/8V_{\horlen}\bigr).
\end{align}
\endgroup
\end{theorem}
\end{myOCP}

\begin{remark}
    \label{rem:interpretation-theorem}
    We note \eqref{eq:azuma:bound:1} indicates that the probability that there exists a time \(s \in \lcrc{1}{\horlen}\) such that \(S_s\) violates the threshold \(h\) under (H0) and consequently, the detector raises a false alarm at some time \(s\) before \(\horlen\) is upper-bounded by \(2\exp\bigl(-h^2/8V_{\horlen}\bigr)\). Thus, \eqref{eq:azuma:bound:1} is a \textbf{finite horizon false alarm probability}, and therefore only remains valid for finite \(\horlen\).
\end{remark}

\begin{remark}
Under \(\PP_1\), (\ref{eq:assum:drift}) imposes a non-positive drift condition on the score increments: conditionally on the past, \(X_i\) should not, on average, drive the CUSUM statistic upward. This is because, in the no-attack regime, the detector should not tend to cross a suitably chosen threshold. Assumption~\ref{eq:assum:exp:bound} is a \emph{conditional sub-Gaussian} requirement on the centered increment \(\xi_i=X_i-\mu_i\), and controls the random fluctuations of the score around its predictable drift \(\mu_i\). 
The CUSUM sequence \((S_t)_{t \in \Nz}\) is compatible with standard preprocessing techniques such as saturation or clipping, which make the score function uniformly bounded and thereby enforce (\ref{eq:assum:exp:bound}) by construction \cite{ref:huber1965robust}. Consequently, (\ref{eq:assum:drift}) and (\ref{eq:assum:exp:bound}) remain general enough to accommodate a broad class of score functions \eqref{eq:score:function}.
\hfill \(\vardiamond\)
\end{remark}

\begin{proof}[Proof of Theorem \ref{thrm:cusum:azuma-type:bounds:new}]
Define \(A_t \Let \sum_{i=1}^t \increment_i\) with \(A_0 \Let 0\). Then, the CUSUM identity can be written as
\begin{align}\label{eq:St:to:At}
    \testsignal_t = A_t - \min_{0\le k \le t} A_k = \max_{0 \le k \le t}\bigl( A_t - A_k \bigr). 
\end{align}
To see this, define \(B_t \Let A_t - \min_{0 \le k \le t} A_k\) with \(B_0 = 0\). Since \(\min_{0\le k\le t}A_k=\min\{\min_{0\le k\le t-1}A_k,\ A_t\},\) we get \(B_t = \max \aset[\big]{ A_t- \min_{0\le k\le t-1}A_k, 0}= \max \aset[\big]{B_{t-1} + \increment_t, 0},\) which is exactly the CUSUM recursion with \(S_0=0\). Hence \(B_t=S_t\). Let us now define the \(\filtration_{t-1}\)-measurable one step ahead conditional mean of \(\increment_t\) and the innovation or centered increment, respective, by \(\mu_i \Let \EE_1[\increment_i \mid \filtration_{i-1}]\) and \(\xi_i \Let \increment_i - \mu_i.\) Note that, since \(\EE_1[\increment_i] < +\infty\) we have \(\EE_1[\mu_i] < +\infty\); moreover, \(\xi_i\) has zero conditional mean, i.e., \(\EE_1[\xi_i \mid \filtration_{i-1}] =  \EE_1[\increment_i-\mu_i \mid \filtration_{i-1}]\). Now consider the cumulative innovation sequence \(M_t \Let \sum_{i=1}^t \xi_i\) and the predictable process \(D_t \Let \sum_{i=1}^t \mu_i.\)
Note that \(M_t\) is a finite sum of integrable \(\filtration_i\)-measurable terms, so it is \(\filtration_t\)-measurable and integrable. Moreover, we have 
\begin{align}
\EE_1[M_t\mid \filtration_{t-1}]
& = \EE_1\Bigl[M_{t-1}+\bigl(\increment_t -\mu_t\bigr)\Big{|}\filtration_{t-1}\Bigr] \nn \\& = M_{t-1}+ \EE_1[\increment_t \mid \filtration_{t-1}] - \mu_t
= M_{t-1}. \nn
\end{align}
Thus, \((M_t,\filtration_t)\) is a Martingale. We show that \(\bigl(D_t\bigr)_{t \ge 1}\) is a predictable process \cite[Section 1.4]{ref:Shreve:Karatzas}: for \(t\ge 1\), we have \(D_t=\sum_{i=1}^t \mu_i\), and each \(\mu_i\) is \(\filtration_{i-1}\)-measurable. Because \(\filtration_{i-1}\subseteq \filtration_{t-1}\) when \(i\le t\), the sum is \(\filtration_{t-1}\)-measurable. Hence \((D_t)_{t \ge 0}\) is predictable, i.e., \(D_t\) is \(\filtration_{t-1}\)-measurable for \(t\ge1\). Moreover it is integrable: indeed, we have \(\EE_1[\mu_i| \le \EE_1 \bigl[\abs{X_i}\bigr]< + \infty\) and consequently, for any finite horizon \(t\), \(\EE_1\bigl[\abs{D_t}\bigr] \le\ \sum_{i=1}^t \EE_1\bigl[\abs{\mu_i}\bigr] \le\ \sum_{i=1}^t \EE_1\bigl[\abs{X_i}\bigr] < + \infty.\) Finally, by definition \(D_0=0\). Thus, all the conditions for the Doob's decomposition theorem \cite[Theorem 4.10]{ref:Shreve:Karatzas} are met and we have \(A_t = M_t +D_t\).

Also, due to Assumption \eqref{eq:assum:drift}, \(\mu_i \le 0\) a.s. for all \(i\) and \(D_u - D_k = \sum_{i=k+1}^{u} \mu_i \le 0\) a.s. for \(u \le t\), and thus, \((D_u)_{u \in \Nz}\) is non-increasing sequence.
Now, using \eqref{eq:St:to:At} and \(A_s - A_k=(M_u - M_k)+(D_u - D_k)\) along with the preceding arguments, we see that for each \(u \le t\),
\begin{align}\label{eq:martingal:dominance}
    S_u = \max_{0\le k\le u}(A_u - A_k) \le \max_{0\le k\le u}\bigl(M_u - M_k\bigr) \text{ a.s.} 
\end{align}
This implies that 
\begin{align*}
    \max_{1 \le u \le t} S_u \le \max_{1 \le u \le t} \max_{0\le k\le u}\bigl(M_u - M_k\bigr) \text{ a.s.}
\end{align*}
such that the inclusion holds for any \(h>0\):
\begin{align*}
    \aset[\Big]{\max_{1 \le u \le t} S_u \ge h} \subseteq \aset[\bigg]{\max_{1 \le u \le t} \max_{0\le k\le u}\bigl(M_u - M_k\bigr) \ge h},
\end{align*}
and \(M_u - M_k \le \max_{1 \le s \le t} M_s + \max_{1 \le s \le t} (-M_s)\) is valid for any indices \(u, k \in \Nz\).
Using the fact that for any real numbers \(x,y\) and \(h>0\), \(x-y\ge h\) implies \(x\ge h/2\) or \(-y\ge h/2\) we see that 
\begin{align}\label{eq:aux-step}
    &\aset[\bigg]{\max_{1 \le u \le t} \max_{0\le k\le u}\bigl(M_u - M_k\bigr) \ge h} \nn\\
    &\subseteq \aset[\Big]{\max_{1\le s\le t} M_s\ge \tfrac{h}{2}} \cup \aset[\Big]{\max_{0\le s\le t} (-M_s)\ge \tfrac{h}{2}}
\end{align}
Combining \eqref{eq:martingal:dominance} and \eqref{eq:aux-step} together, we finally get the set-theoretic inclusion
\begin{align}\label{eq:set:bound}
 &\aset[\Big]{\max_{1 \le u \le t} S_u \ge h}\nn\\ &\subseteq 
\aset[\Big]{\max_{1\le s\le t} M_s\ge \tfrac{h}{2}} \cup \aset[\Big]{\max_{0\le s\le t} (-M_s)\ge \tfrac{h}{2}}.
\end{align}

We now show that, for \(a>0\) the concentration bound holds:
\begin{align}\label{eq:final:half:bound}
\PP_1 \Bigl(\max_{1\le s\le t} M_s \ge a \Bigr) \le \exp\bigl( -a^2/2V_t\bigr).
\end{align}
Then we will apply the same bound to \(-M_s\). To this end, fix \(\lambda>0\) and define the process \(Z_s \Let \exp \bigl( \lambda M_s - \frac{\lambda^2}{2}V_s \bigr),\) for \(s=0,1,\dots,t,\) with \(M_0=0\), \(V_0=0\), hence \(Z_0=1\). We show that \((Z_s,\filtration_s)\) is a nonnegative supermartingale.  Indeed, writing the martingale increment: \(M_s = M_{s-1} + \xi_s\) and \(V_s = V_{s-1} + \sigma_s^2\), we see that \(Z_s = \exp \Bigl( \lambda(M_{s-1}+\xi_s) - \frac{\lambda^2}{2}(V_{s-1}+\sigma_s^2)\Bigr) = Z_{s-1}\cdot \exp \Bigl(\lambda\xi_s-\frac{\lambda^2}{2}\sigma_s^2\Bigr).\) Taking conditional expectation w.r.t \(\filtration_{s-1}\) and recalling that \(Z_{s-1}\) is \(\filtration_{s-1}\)-measurable, we get
\begin{align}
 \EE_1[Z_s\mid \filtration_{s-1}] &= Z_{s-1}
\EE_1 \Bigl[\exp \Bigl(\lambda\xi_s-\frac{\lambda^2}{2}\sigma_s^2\Bigr)\Bigm|\filtration_{s-1}\Bigr] \nn \\& = Z_{s-1} \epower{-\frac{\lambda^2}{2}\sigma_s^2}
\EE_1 \Bigl[ \epower{\lambda\xi_s}\mid \filtration_{s-1}\Bigr] \le Z_{s-1}, \nn
\end{align}
where in the last step we employed \eqref{eq:assum:exp:bound}. Thus
\((Z_s,\filtration_s)\) is a nonnegative supermartingale.

We also note that using monotonicity of the exponential and using the fact that \(V_s \le V_t\) for \(s \le t\), the event \(\aset[]{\exists\ s\le t \suchthat\ M_s\ge a}\) implies \(\aset[]{\exists\ s\le t \suchthat  Z_s\ge c}\) where \(
c \Let \exp \Bigl(\lambda a - \frac{\lambda^2}{2}V_t\Bigr)\). Then we have the inclusion
\begin{align}\label{eq:set:bound2}
\aset[\Big]{\max_{1\le s\le t} M_s\ge a} \subseteq
\aset[\Big]{\max_{1\le s\le t} Z_s\ge c}.    
\end{align}
Define the hitting time \(\tau \Let \inf \aset[]{1\le s\le t \suchthat  Z_s\ge c}\), with the convention \(\tau= +\infty\) if no such \(s\) exists. Then \(\aset[]{\max_{1\le s\le t} Z_s\ge c }=\aset[]{\tau\le t}\). Define the bounded stopping time \(\tau' \Let \min\{\tau,t\}\) so that \(\tau'\le t\) always. Since \(Z_s\) is a nonnegative supermartingale and \(\tau'\) is bounded, one has \(\EE_1[Z_{\tau'}]\le \EE_1[Z_0]=1\). Moreover, on the event \(\aset[]{\tau\le t}\) we have \(\tau'=\tau\) and \(Z_{\tau'}=Z_{\tau}\ge c\). Hence \(Z_{\tau'} \ge c \indic{\aset[]{\tau\le t}}.\) Taking expectation and recalling the fact that \(\EE_1[Z_{\tau'}]\le 1\), we obtain 
\begin{align}\label{eq:Zs:bound}
\PP_1 \bigl( \max_{1\le s\le t} Z_s\ge c \bigr) \le (1/c).
\end{align}
From \eqref{eq:set:bound2} and \eqref{eq:Zs:bound} we have
\begin{align}\label{eq:set:prob:bound}
\PP_1 \Bigl( \max_{1\le s\le t} M_s\ge a \Bigr)
&\le \PP_1 \Bigl( \max_{1\le s\le t} Z_s\ge c \Bigr)
\le (1/c).  
\end{align}
Optimizing the right-hand side of \eqref{eq:set:prob:bound} with respect to \(\lambda>0\) gives \(\bar{\lambda} = \frac{a}{V_t}\) and this yields the required bound \eqref{eq:final:half:bound}.

The identical argument applied to \(-M_s\) (noting that \eqref{eq:assum:exp:bound} holds for all \(\lambda\in \Rbb\), hence also for \(-\xi_i\)) gives
\begin{align}\label{eq:final:other:half:bound}
\PP_1 \Bigl( \max_{0\le s\le t}(-M_s)\ge a \Bigr)
\le \exp ( -a^2/2V_t) .
\end{align}
Finally, \eqref{eq:final:half:bound}--\eqref{eq:final:other:half:bound} (with \(a=h/2\)) together with \eqref{eq:set:bound} gives us 
\begin{align}\label{eq:final bound}
&\PP_1 \Big(\max_{1 \le u \le t} S_u \ge h\Big)\nn\\
& \le \PP_1 \Bigl(\max_{1\le s\le t} M_s \ge h/2 \Bigr) + \PP_1 \Big(\max_{0\le s\le t}(-M_s)\ge h/2\Big)\nn\\
& \le 2\exp \Bigl(-\frac{h^2/4}{2 V_t} \Bigr)
=2 \exp \Bigl(-\frac{h^2}{8V_t}\Bigr). 
\end{align}
The proof is complete.
\end{proof}

We provide a \(\text{(H1)}\)-side guarantee as well following \cite{ref:YHH-VMV-24}.

\begin{myOCP}
\begin{theorem}[Post-attack detection guarantee]\label{thm:post:attack:delay}
Fix \(h>0\) and an attack time \(\nu\in\N\). On \((\samplesp,\filtration)\), let
\(\filtration_0\Let\{\varnothing,\samplesp\}\) and \(\filtration_i\Let\sigma(\increment_1,\ldots,\increment_i)\), \(i\in\N\), be the natural filtration, and define the detection time
\begin{align*}
 \tau_h\Let\inf \aset[]{t\in\N \suchthat \testsignal_t\ge h },
\end{align*}
with the convention that \(\inf\varnothing=+\infty\), where \((\testsignal_t)_{t\in\N}\) is defined in \eqref{eq:cusum:iteration}. Assume that, under \(\PP_{\nu}\), the probability measure corresponding to an attack initiated at time \(\nu\), there exist constants \(\theta>0\) and \(c_{\theta}>0\) such that, for every \(i\ge\nu\),
\begin{align}\label{eq:assum:post:attack:exp}
\EE_{\nu}\left[\exp\bigl(-\theta\increment_i\bigr)
\,\middle|\,\filtration_{i-1}\right]\le\exp(-c_{\theta})
\quad\text{a.s.}
\end{align}
Then, for every \(d\in\N\), we have the bound
\begin{align}\label{eq:post:attack:delay:bound}
\PP_{\nu}\bigl(\tau_h-\nu\ge d\,\big|\,\tau_h\ge\nu\bigr)\le
\min\left\{1,\,\exp\bigl(\theta h-dc_{\theta}\bigr)
\right\},
\end{align}
provided that \(\PP_{\nu}(\tau_h\ge\nu)>0\). Consequently, for every prescribed missed-detection tolerance \(\delta_{\mathrm D}\in(0,1)\), any
\begin{align}\label{eq:post:attack:delay:threshold}
d \ge\left\lceil\frac{\theta h+\log(1/\delta_{\mathrm D})}{c_{\theta}}
\right\rceil
\end{align}
ensures that \(\PP_{\nu}\bigl(\tau_h-\nu\ge d\,\big|\,\tau_h\ge\nu
\bigr)\le\delta_{\mathrm D}.\) Moreover, the conditional average detection delay satisfies
\begin{align}
\EE_{\nu}\bigl[\tau_h-\nu\,\big|\,\tau_h\ge\nu\bigr]\le\left\lceil
\frac{\theta h}{c_{\theta}}\right\rceil+\frac{1}{\exp(c_{\theta})-1}. \hspace{3mm} \vardiamond \nn
\end{align}
\end{theorem}
\end{myOCP}

\begin{proof}
Define \(A_{\nu}\Let\aset[]{\tau_h\ge\nu}.\) Note that \(A_{\nu}\in\filtration_{\nu-1}\), since it is the event that the detector does not cross the threshold before \(\nu\). Fix \(d\in\N\), and define the cumulative post-attack score
\begin{align*}
R_{\nu,d}\Let\sum_{i=\nu}^{\nu+d-1}\increment_i.
\end{align*}
By the CUSUM identity \eqref{eq:St:to:At}, we have
\begin{align*}
\testsignal_{\nu+d-1}=\max_{0\le k\le\nu+d-1}\sum_{i=k+1}^{\nu+d-1}\increment_i\ge\sum_{i=\nu}^{\nu+d-1}\increment_i=R_{\nu,d},
\end{align*}
where the inequality follows by choosing \(k=\nu-1\). Moreover, on the event \(\aset[]{\tau_h\ge\nu+d}\), the detector has not crossed the threshold by time \(\nu+d-1\), and hence \(\testsignal_{\nu+d-1}<h\). Consequently,
\begin{align}\label{eq:post:attack:event:inclusion}
\aset[]{\tau_h\ge\nu+d}\subseteq A_{\nu}\cap\aset[]{R_{\nu,d}<h}.
\end{align}

We next establish the conditional exponential-moment bound
\begin{align}\label{eq:post:attack:sum:mgf}
\EE_{\nu}\left[\exp\bigl(-\theta R_{\nu,d}\bigr)\,\middle|\,\filtration_{\nu-1}\right]\le\exp(-dc_{\theta}).
\end{align}
Indeed, by the tower property and the fact that \(\exp\bigl(-\theta\sum_{i=\nu}^{\nu+d-2}\increment_i\bigr)\) is \(\filtration_{\nu+d-2}\)-measurable, we have
\begin{align*}
&\EE_{\nu}\left[\exp\left(-\theta\sum_{i=\nu}^{\nu+d-1}\increment_i\right)\,\middle|\,\filtration_{\nu-1}\right]\\
&=\EE_{\nu}\Bigg[\begin{aligned}
&\exp\left(-\theta\sum_{i=\nu}^{\nu+d-2}\increment_i\right)\\
&\qquad\times\EE_{\nu}\left[\exp\bigl(-\theta\increment_{\nu+d-1}\bigr)\,\middle|\,\filtration_{\nu+d-2}\right]
\end{aligned}\,\Bigg|\,\filtration_{\nu-1}\Bigg]\\
&\le\exp(-c_{\theta})\EE_{\nu}\left[\exp\left(-\theta\sum_{i=\nu}^{\nu+d-2}\increment_i\right)\,\middle|\,\filtration_{\nu-1}\right].
\end{align*}
where the inequality follows from \eqref{eq:assum:post:attack:exp}. Repeating this argument for the remaining \(d-1\) increments yields \eqref{eq:post:attack:sum:mgf}.
Since \(\theta>0\), we have
\begin{align*}
\aset[]{R_{\nu,d}<h}=\aset[]{\exp(-\theta R_{\nu,d})>\exp(-\theta h)}.
\end{align*}
Thus, the conditional Markov inequality and \eqref{eq:post:attack:sum:mgf} give
\begin{align}\label{eq:post:attack:chernoff}
\PP_{\nu}\left(R_{\nu,d}<h\,\middle|\,\filtration_{\nu-1}\right)&\le\exp(\theta h)\EE_{\nu}\left[\exp(-\theta R_{\nu,d})\,\middle|\,\filtration_{\nu-1}\right]\nn\\
&\le\exp\bigl(\theta h-dc_{\theta}\bigr).
\end{align}
Using \eqref{eq:post:attack:event:inclusion}, the fact that \(A_{\nu}\in\filtration_{\nu-1}\), and \eqref{eq:post:attack:chernoff}, we obtain
\begin{align*}
\PP_{\nu}\left(\tau_h\ge\nu+d\,\middle|\,\filtration_{\nu-1}\right)&\le\indic{A_{\nu}}\PP_{\nu}\left(R_{\nu,d}<h\,\middle|\,\filtration_{\nu-1}\right)\\
&\le\indic{A_{\nu}}\exp\bigl(\theta h-dc_{\theta}\bigr)\quad\text{a.s.}
\end{align*}
Taking expectations on both sides yields
\begin{align*}
\PP_{\nu}(\tau_h\ge\nu+d)\le\PP_{\nu}(A_{\nu})\exp\bigl(\theta h-dc_{\theta}\bigr).
\end{align*}
Since \(A_{\nu}=\aset[]{\tau_h\ge\nu}\) and \(\PP_{\nu}(\tau_h\ge\nu)>0\), we obtain
\begin{align*}
\PP_{\nu}\bigl(\tau_h-\nu\ge d\,\big|\,\tau_h\ge\nu\bigr)\le\exp\bigl(\theta h-dc_{\theta}\bigr).
\end{align*}
Combining this estimate with the trivial upper bound \(1\) proves \eqref{eq:post:attack:delay:bound}. For any \(\delta_{\mathrm D}\in(0,1)\), the condition
\begin{align*}
d\ge\left\lceil\frac{\theta h+\log(1/\delta_{\mathrm D})}{c_{\theta}}\right\rceil
\end{align*}
implies \(\theta h-dc_{\theta}\le-\log(1/\delta_{\mathrm D})=\log(\delta_{\mathrm D}),\)
and hence \(\exp\bigl(\theta h-dc_{\theta}\bigr)\le\delta_{\mathrm D}.\) This proves the missed-detection guarantee.

Finally, set \(m\Let\left\lceil\tfrac{\theta h}{c_{\theta}}\right\rceil.\)
Since \(\tau_h-\nu\) is a nonnegative integer-valued random variable on \(A_{\nu}\), the tail-sum formula and \eqref{eq:post:attack:delay:bound} yield
\begin{align*}
&\EE_{\nu}\bigl[\tau_h-\nu\,\big|\,\tau_h\ge\nu\bigr]=\sum_{d=1}^{+\infty}\PP_{\nu}\bigl(\tau_h-\nu\ge d\,\big|\,\tau_h\ge\nu\bigr)\\
&\le\sum_{d=1}^{m}1+\sum_{d=m+1}^{+\infty}\exp\bigl(\theta h-dc_{\theta}\bigr)\\
&=m+\frac{\exp\bigl(\theta h-(m+1)c_{\theta}\bigr)}{1-\exp(-c_{\theta})}=m+\frac{\exp\bigl(\theta h-mc_{\theta}\bigr)}{\exp(c_{\theta})-1}.
\end{align*}
By the definition of \(m\), we have \(mc_{\theta}\ge\theta h\), and therefore \(\exp\bigl(\theta h-mc_{\theta}\bigr)\le 1.\) Consequently,
\begin{align*}
\EE_{\nu}\bigl[\tau_h-\nu\,\big|\,\tau_h\ge\nu\bigr]\le\left\lceil\frac{\theta h}{c_{\theta}}\right\rceil+\frac{1}{\exp(c_{\theta})-1}.
\end{align*}
The proof is complete.
\end{proof}

\begin{remark}
Observe that Theorem \ref{thrm:cusum:azuma-type:bounds:new} and Theorem \ref{thm:post:attack:delay} quantify the effects of the threshold \(h\). For prescribed tolerances \(\eta_{\mathrm{FA}},\eta_{\mathrm{D}}\in \loro{0}{1}\), the finite-horizon false-alarm guarantee requires 
\[
h\ge \sqrt{8V_T\log\left(\frac{2}{\eta_{\mathrm{FA}}}\right)},
\] 
which is obtained from Theorem  \ref{thrm:cusum:azuma-type:bounds:new}), whereas detection within \(d\) post-attack samples with missed-detection probability at most \(\eta_{\mathrm{D}}\) is guaranteed if 
\[
h\le \frac{dc_{\theta}-\log(1/\eta_{\mathrm{D}})}{\theta}
\]
which is obtained from Theorem \ref{thm:post:attack:delay}. Hence, both guarantees hold whenever
\begin{align}
\sqrt{8V_T\log\left(\frac{2}{\eta_{\mathrm{FA}}}\right)}
\leq h\le \frac{dc_{\theta}-\log(1/\eta_{\mathrm{D}})}{\theta}. \nn
\end{align}
Thus, increasing \(h\) improves finite-horizon false-alarm control but worsens the finite-window missed-detection and delay guarantees. In particular, when \(c_{\theta}\) is small, a larger detection window \(d\) is required for the two guarantees to be simultaneously feasible.
The threshold used in the numerical study is calibrated from nominal data to control false alarms. Theorem \ref{thm:post:attack:delay} theoretically complements this calibration by quantifying the resulting post-attack missed-detection probability and average detection delay, thereby guiding the designer in selecting a threshold that balances these competing objectives.
\end{remark}

\section{Numerical validation and discussion}\label{sec:num_exp}
We compare the performance of our detector with that of the CUSUM detector \cite{ref:AN-AT-AA-SD-2023} in Gaussian settings and provide positive results in certain non-Gaussian regimes. 
We note that a a direct numerical comparison with \cite{ref:DL:SM:LCSS:WassDetection,ref:VR:NH:JR:TS:WassAnomaly,ref:YF:DL:CS:WassDetection:FDI} would not be fair, as these works solve materially different detection problems. Specifically, \cite{ref:VR:NH:JR:TS:WassAnomaly} tunes a fixed quadratic detector using a nominal moment-based ambiguity set, \cite{ref:DL:SM:LCSS:WassDetection} uses only nominal data and a sliding-window Wasserstein statistic, and \cite{ref:YF:DL:CS:WassDetection:FDI} assumes structured attack channels and a parity-space detector with a reachability-based design objective. In contrast, we do not require a parametric attack model, uses residual data from both nominal and attacked regimes, computes the WCDs through robust binary hypothesis testing, and applies the resulting score sequentially through CUSUM. Implementing \cite{ref:DL:SM:LCSS:WassDetection,ref:VR:NH:JR:TS:WassAnomaly,ref:YF:DL:CS:WassDetection:FDI} in our setting would therefore require substantial reformulation and non-equivalent tuning.

We considered the benchmark linearized quadruple-tank process \cite[Eq. \(1\)]{ref:KHJ-02}:
{\footnotesize\begin{align*}
    A = \begin{pmatrix}
        0.968 & 0 & 0.082 & 0\\
        0 & 0.978 & 0 & 0.064\\
        0 & 0 & 0.917 & 0\\
        0 & 0 & 0 & 0.935
    \end{pmatrix}, \, B = \begin{pmatrix}
        0.164 & 0.004\\0.002 & 0.124\\0 & 0.092\\0.06 & 0
    \end{pmatrix},
\end{align*}}
with \(C = \identity{4}\) (the \((4 \times 4)\) identity matrix).

\begin{figure*}
\centering
\begin{subfigure}{0.45\linewidth}
    \includegraphics[scale = 0.14]{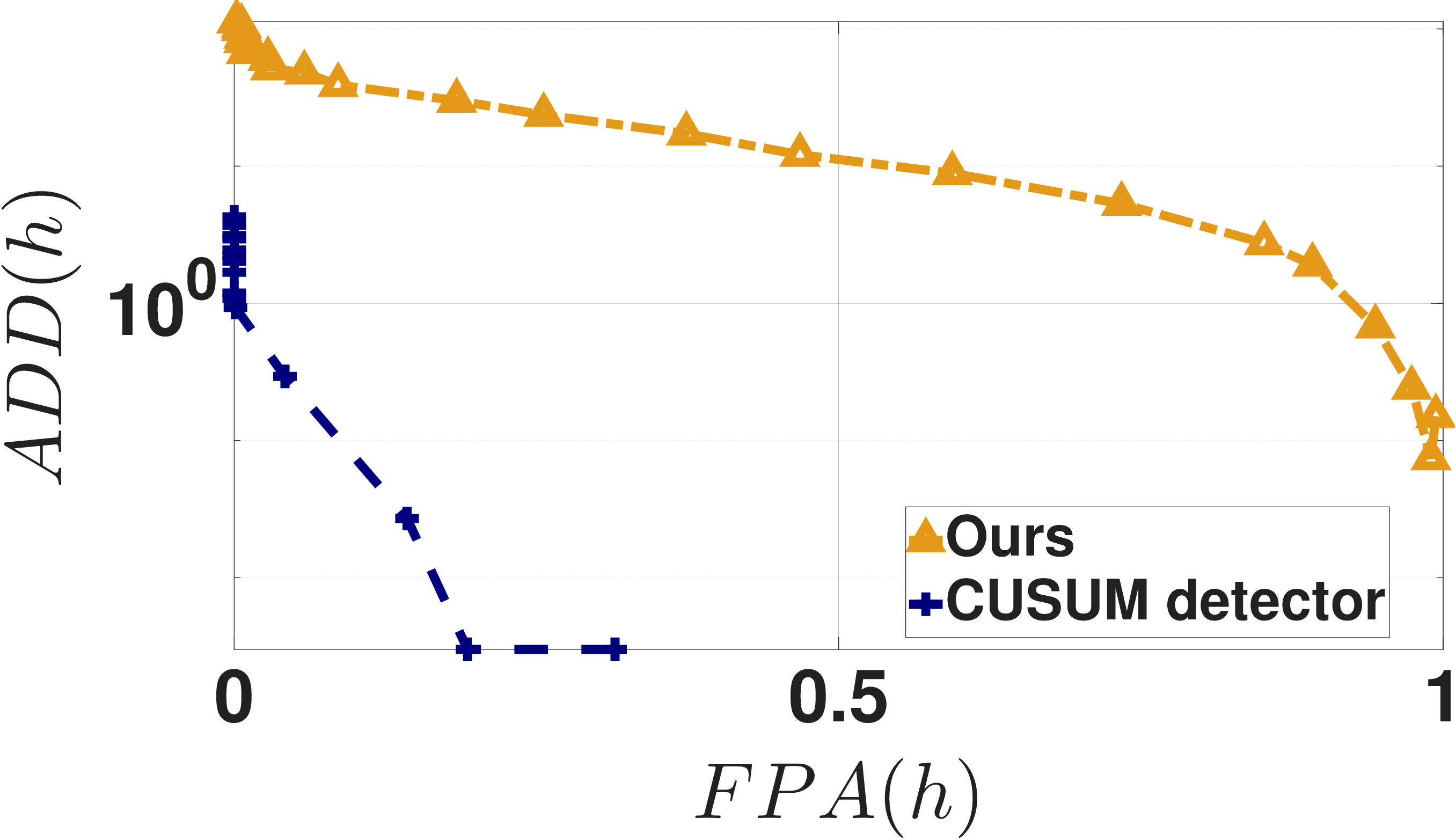}
    \label{subfig:gaus_att_15}
\end{subfigure} 
\begin{subfigure}{0.45\linewidth}
    \includegraphics[scale = 0.14]{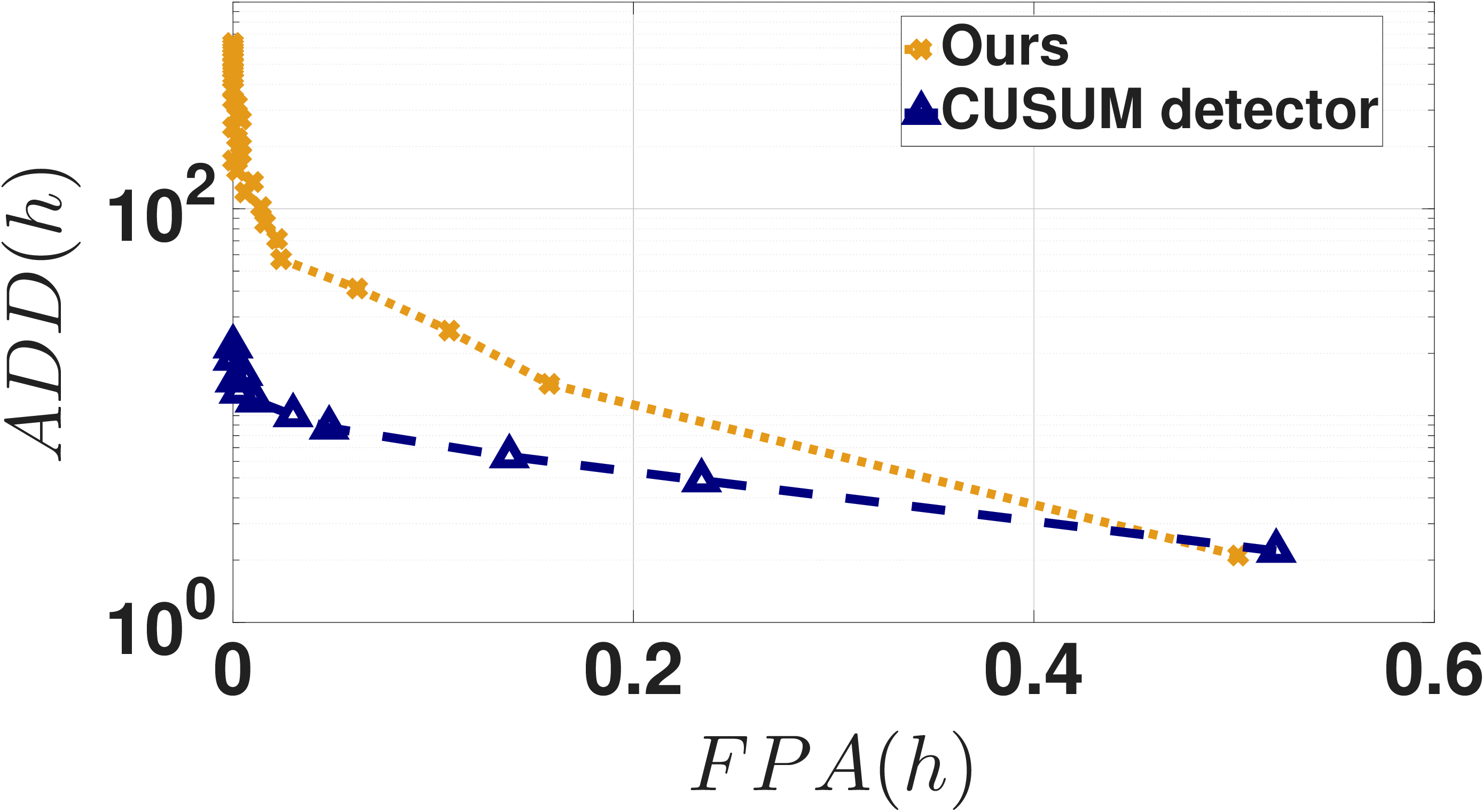}
    \label{subfig:gaus_att_25}
\end{subfigure} 
\begin{subfigure}{0.45\linewidth}
    \includegraphics[scale = 0.14]{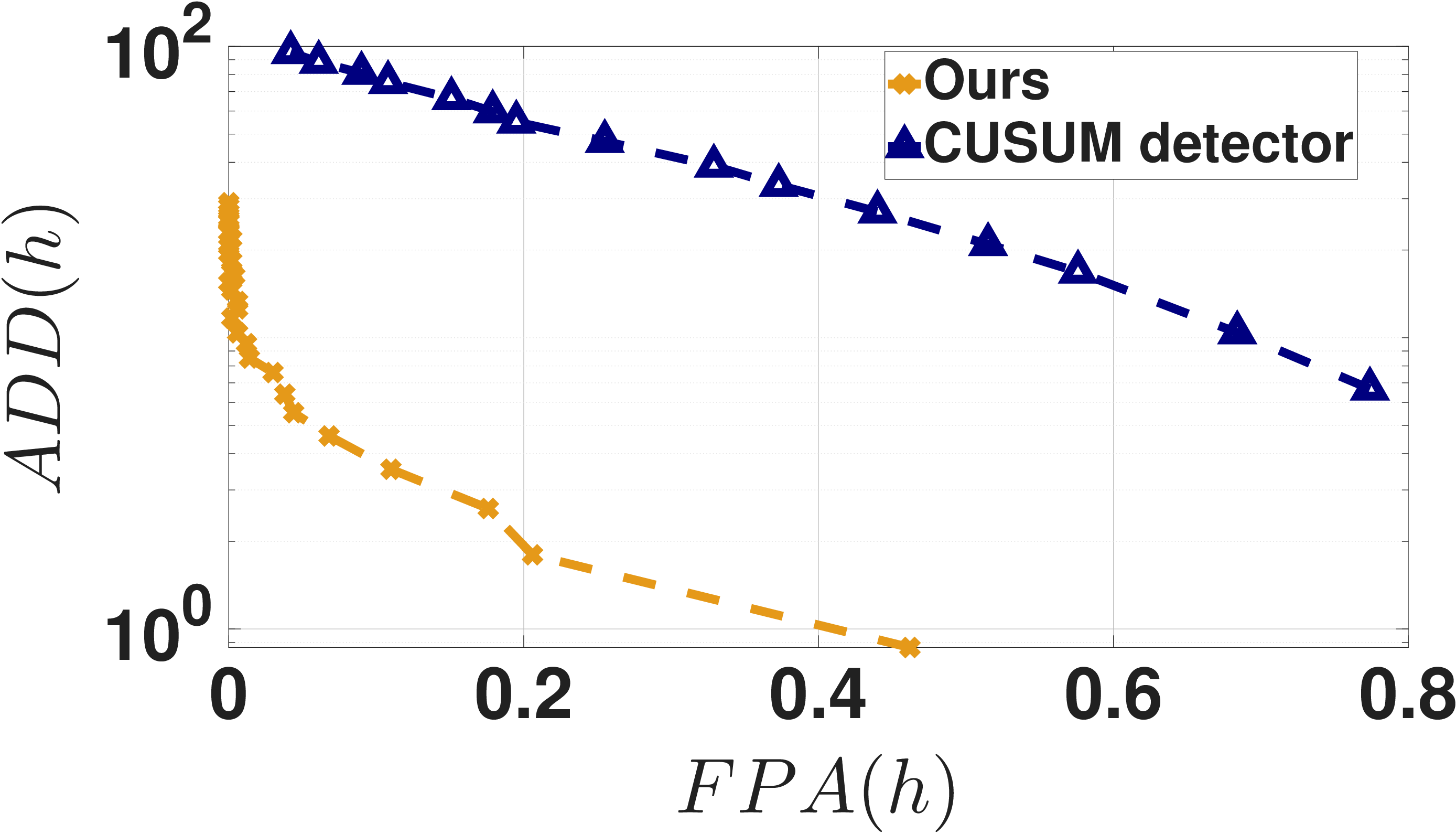}
    \label{subfig:gaus_att_15}
\end{subfigure} 
\begin{subfigure}{0.45\linewidth}
    \includegraphics[scale = 0.14]{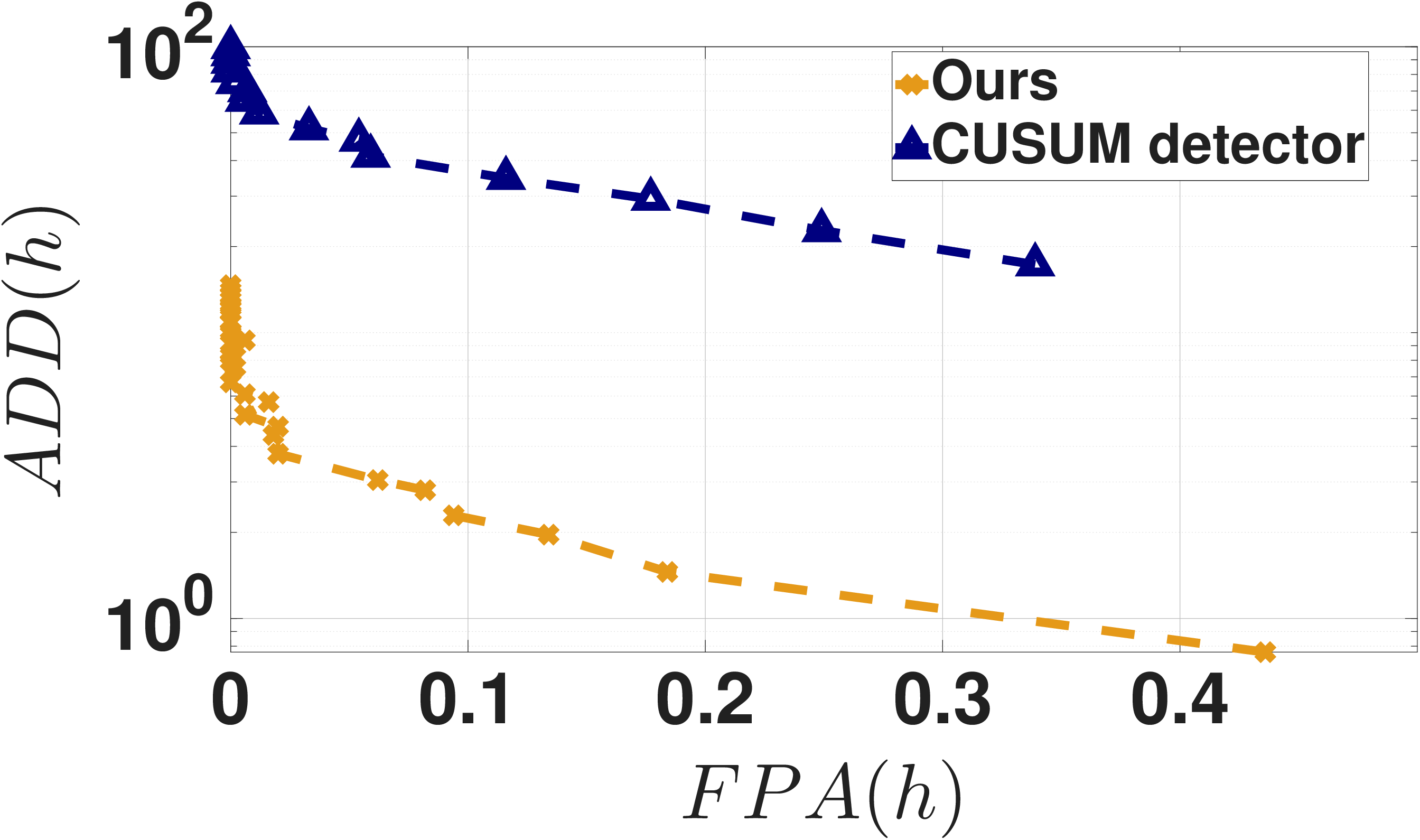}
    \label{subfig:gaus_att_25}
\end{subfigure} 
\caption{Comparison with \cite{ref:AN-AT-AA-SD-2023} for i.i.d. attacks $\widehat{v}_t\sim\normal(0,\attvar\identity{\dimout})$, with $\attvar=0.1,0.5,1.5,2.5$ shown from top-left to bottom-right.}
\label{fig:gaus_att}
\end{figure*}

\noindent\textbf{Performance metrics:} Let \(\tmatt\) denote the unknown but deterministic time when the attack was initiated, and \(\tmdet\) be the time of attack detection, defined by \(\tmdet \Let \inf \aset[\big]{t \in \Nz \suchthat S_t\ge h}\) for a fixed threshold \(h\). To choose a practical threshold \(h\), we estimate the false positive probability, given by \(\PP_1\bigl(\max_{1 \le s \le \horlen} \testsignal_{s}\ge h \bigr)\) finite \(\horlen\), empirically via conducting Monte Carlo simulations for different threshold values \(h\). The threshold corresponding to the desired false positive tolerance \(\eta\) is then selected for deployment. We considered the following standard metrics to assess the performance: 
\begin{itemize}[leftmargin=*, label = \(\circ\)]

    \item \emph{Average detection delay (ADD):} For a fixed threshold \(h\), ADD is defined by \(\ADD(h) \Let \EE_2 \cexpecof[\big]{\tmdet - \tmatt \given \tmdet \ge \tmatt},\) where the expectation is defined with respect to the probability measure under (H1), and \(h\) is the threshold.
    
    \item \emph{False alarm rate (FAR):} Fix \(h\), and compute FAR as \(\FAR(h) \Let \PP_2 \cprobof[\big]{\tmdet < \tmatt},\) where the probability measure is defined with respect to (H1).   
\end{itemize}
\vspace{1mm}
We examined \(\ADD\) with respect to \(\FAR\) for different values of \(h\), which reveals how quickly the attack is detected for a fixed \(\FAR(h)\). Monte-Carlo simulations were conducted to compute \(\ADD(\cdot)\) and \(\FAR(\cdot)\).
 
\noindent\textbf{Attack model:} We considered the deception attack \cite{ref:AN-AT-AA-SD-2023} described by \(v_{t} = A_a v_{t-1} + \widehat{v}_t\), where \(A_a  = \mathrm{diag}(4,2)\). The sequence \((\widehat{v}_t)_{t \in \Nz}\) refers to the uncertainties associated with the adversary. We considered the following two cases: \textbf{(a)} \((\widehat{v}_t)_{t \in \Nz} \sim \mathcal{N}(0, \Sigma_a)\) is a sequence of \(d_y\)-dimensional random vectors; \textbf{(b)} \((\widehat{v}_t)_{t \in \Nz} \sim \mathcal{N}(0, \Sigma_a)\) and further corrupted by the exponential distribution \(\mathrm{Exp}(\lambda)\), where \(\lambda>0\). It is assumed that during the attack, the adversary, with pre-defined \((A_a, \Sigma_a)\), may \emph{completely replace the true output signal} \(y\) with the corrupt data generated by \(v\). We fixed \(\tmatt = 250\) seconds.

\noindent\textbf{Parameters:}
We set \(\sigma = 0.5\) in \S\ref{subsec:kernel}. For reproducibility, we fixed the seed to be \(2\). In Fig.~\ref{fig:gaus_att} and \ref{fig:G_exp_att}-\ref{fig:on_support}, we chose \((\eps_1, \eps_2)\) to be \((0.001, 0.001)\) and \((0.001, 0.01)\), respectively. We set \((n_1, n_2)\) to be \((150,100)\) for all figures.

\noindent \textbf{Results and discussions:}
First, we assumed that \(E\dist_t \overset{\text{i.i.d.}}{\sim} \normal(0, 0.1 \identity{\dimst})\), \(F\dist_t \overset{\text{i.i.d.}}{\sim} \normal(0, 0.05 \identity{\dimout})\), and \(\widehat{v}_t \overset{\text{i.i.d.}}{\sim}\normal(0, \attvar \identity{\dimout})\) for every \(t \in \Nz\) and \(\attvar \in \aset[\big]{0.1, 0.5, 1.5, 2.5}\). Fig.~\ref{fig:gaus_att} summarizes the robust performance of our detector against \((v_t)_{t \in \Nz}\) generated by an adversary driven by Gaussian noise. 
We observe that for small values of \(\attvar\), the \(ADD\) of our detector approaches  \cite{ref:AN-AT-AA-SD-2023}. This is because the CUSUM-based detector is optimal for Gaussian noise and attack process \(v\). However, as \(\attvar\) increases further, we observe that the \(\ADD\) of our detector is smaller, indicating a \emph{faster} response to attack-induced change points. 


Second, we tested the robust performance of our detector when  \(\widehat{v}_t \overset{\text{i.i.d.}}{\sim}\normal(0, 0.05 \identity{\dimout}) + \text{Exp}(\lambda)\) for each \(t\in \Nz\), which is non-Gaussian. Here \(\lambda \in \aset[]{0.5, 1.5}\), fixing all other parameters. Fig.~\ref{fig:G_exp_att} reveals that as we increase \(\lambda\), which amplifies the non-Gaussian component, our detector clearly outperforms the CUSUM detector in terms of quickly responding to adversary-induced changes for various \(\FAR\), supporting that \(\algoname\) is distributionally robust.
\begin{figure}[h]
\begin{subfigure}{\linewidth}
\centering
    \includegraphics[scale = 0.14]{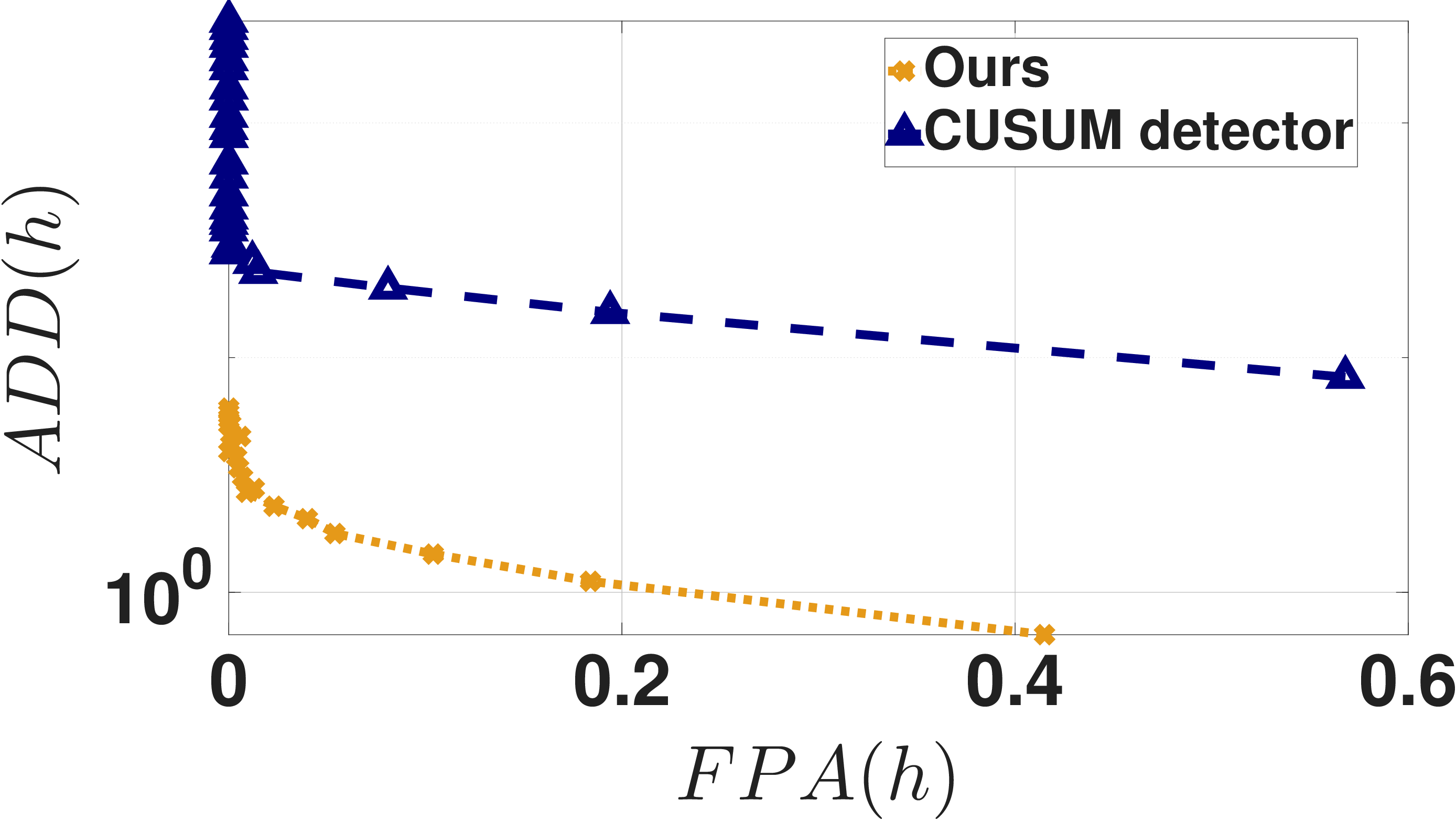}
    \label{subfig:G_exp_att_15}
\end{subfigure} 
\begin{subfigure}{\linewidth}
\centering
    \includegraphics[scale = 0.14]{Figures/new_figures/G_exp_05.eps}
    \label{subfig:G_exp_att_25}
\end{subfigure} 
\caption{Comparison with \cite{ref:AN-AT-AA-SD-2023}, when \(\widehat{v}_t \overset{\text{i.i.d.}}{\sim}\normal(0, \attvar \identity{\dimout}) + \text{Exp}(\lambda)\) for \(\lambda = 0.5\) (left) and \(\lambda = 1.5\) (right).}
\label{fig:G_exp_att}
\end{figure}
%
Figure \ref{fig:on_support} depicts the on-support performance of the detector defined in \eqref{eq:on:supp:test}. The left figure represents the on-support distributions evaluated on \(s_{\ell}\), while the right figure represents the on-support test \(\test^{\star}(\cdot)\) defined in \eqref{eq:on:supp:test}. It is observed that \(P_1^{\star}\) and \(P_2^{\star}\) assign higher probability mass to the non-attacked and attacked samples, respectively. 
\begin{figure}[h]
\begin{subfigure}{\linewidth}
\centering
    \includegraphics[scale = 0.18]{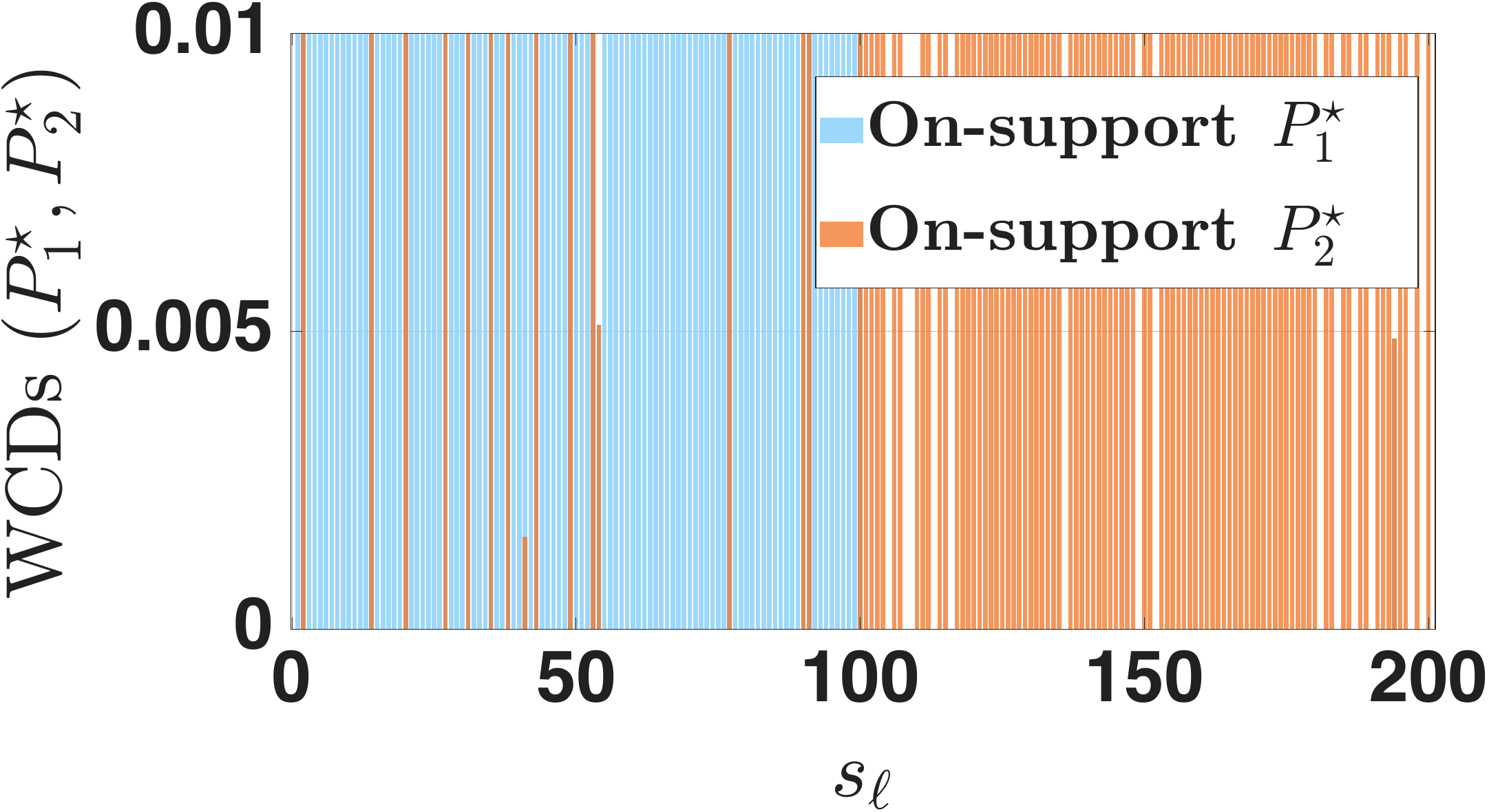}
    \label{subfig:wcds}
\end{subfigure} 
\begin{subfigure}{\linewidth}
\centering
   \hspace{3mm} \includegraphics[scale = 0.17]{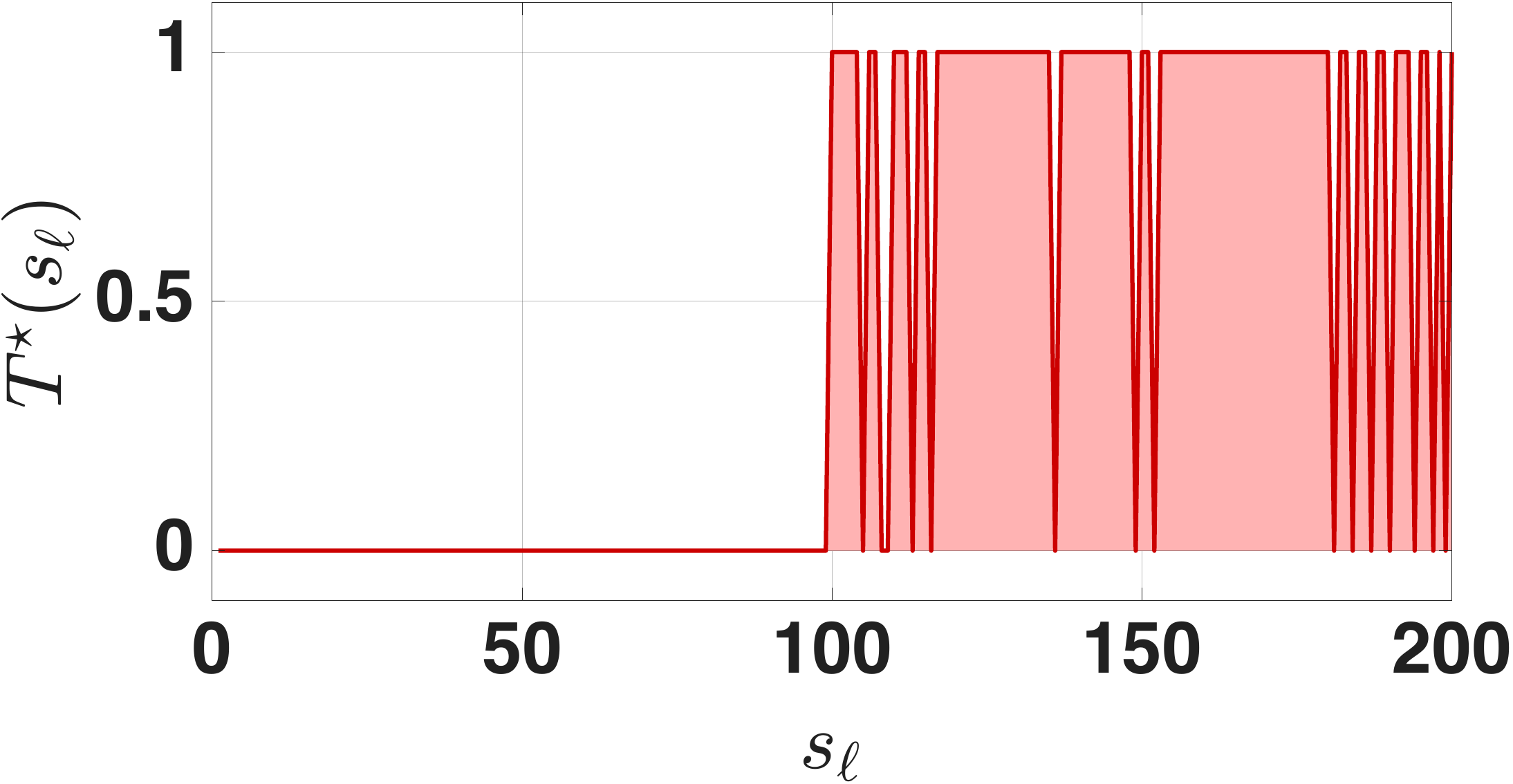}
    \label{subfig:on-support test}
\end{subfigure} 
\caption{On-support WCDs (left)  and \(\test^{\star}(\cdot)\) (right) defined in \eqref{eq:on:supp:test}, when \(\widehat{v}_t \overset{\text{i.i.d.}}{\sim}\normal(0, \attvar \identity{\dimout}) + \text{Exp}(\lambda)\) for \(\lambda = 0.5\). }
\label{fig:on_support}
\end{figure}

To verify Assumption \ref{assum:blanket}, we computed \(\Wdist_1(Q_1,Q_2)\) for every pair of nominal and attacked training datasets used in the numerical experiments; the values are provided in Tables \ref{tab:W_distance_gauss} and \ref{tab:W_distance_nongauss}, which corresponds to Figures \ref{fig:G_exp_att} and \ref{fig:gaus_att}, respectively.  

\begin{table}[ht]
\centering
\begin{subtable}[t]{0.48\columnwidth}
\centering
\begin{tabular}{|c|c|}
\hline
\(\sigma_{\text{a}}\) & \(\Wdist_1(Q_1,Q_2)\) \\ \hline
\(0.1\) & 0.4309 \\ \hline
\(0.5\) & 0.5269 \\ \hline
\(1.5\) & 0.7679 \\ \hline
\(2.5\) & 0.9549 \\ \hline
\end{tabular}
\caption{\(\Wdist_1\)-distance for different values of \(\sigma_{\text{a}}\). For every \(\sigma_{\text{a}}\), \(\eps_1+\eps_2=0.002<\Wdist_1(Q_1,Q_2)\), suggesting that the ambiguity sets are disjoint.}
\label{tab:W_distance_gauss}
\end{subtable}
\hfill
\begin{subtable}[t]{0.48\columnwidth}
\centering
\begin{tabular}{|c|c|}
\hline
\(\lambda\) & \(\Wdist_1(Q_1,Q_2)\) \\ \hline
\(0.5\) & 1.7982 \\ \hline
\(1.5\) & 0.6317 \\ \hline
\end{tabular}
\caption{\(\Wdist_1\)-distance for different values of \(\lambda\) reported in Figure~\red{2}. For every \(\lambda\), \(\eps_1+\eps_2=0.011<\Wdist_1(Q_1,Q_2)\), suggesting that the ambiguity sets are disjoint.}
\label{tab:W_distance_nongauss}
\end{subtable}
\end{table}


\section{Discussion and concluding remarks}
This article presents \(\algoname\), an OT-driven framework for attack detection in CPS. While it relaxes two important restrictions encountered in the literature: attack detection in the Gaussian regime, and the need for attack-model- or policy-specific detection mechanisms, the current exposition provides an \emph{a posteriori} analysis relying on data from both the attacked and non-attacked regimes, and only serves as an initial step toward developing a general framework for the \emph{online detection} of more sophisticated attacks. 

\textbf{Possible applications:} \(\algoname\) can be implemented as a supervisory monitoring layer in partially observed CPSs that already employ an observer or state estimator, such as industrial process-control systems, water-distribution and tank processes, power-system monitoring, and networked transportation or manufacturing systems. 
In these settings, the detector operates directly on the residual sequence produced by the existing observer and therefore does not require modification of the underlying controller or estimator. 
Nominal and representative attacked residual data may be collected from historical operation, controlled experiments, hardware-in-the-loop tests, etc. The Wasserstein-based LP is then solved once during the offline training stage to obtain the WCDs.

\textbf{Next steps:} An immediate direction for future work is to extend the detection mechanism to an online setting while retaining its robustness properties. Furthermore, conducting a more detailed numerical study of various classes of attacks in the non-Gaussian regime is also part of our future research plan.

\section*{Acknowledgment}
\noindent The authors thank Ashwin Aravind, Fujitsu Development Centre, Kawasaki, Japan, for his assistance with some of the simulation results presented in this manuscript. Internal models of ChatGPT and Claude, were used for editing and grammar correction, as well as for generating some supporting helper functions for plotting the figures. The authors take full responsibility for the content of the paper.

\bibliographystyle{ieeetr}
\bibliography{refs}
\newpage

\onecolumn

\appendices
\section{Additional numerical details}
\label{Sec-appen:sensitivity}

We provided some additional background on how to select the hyper parameters \(\sigma\) (the kernel bandwidth) and the pair \((\eps_1,\eps_2)\) (the ambiguity set radii). 

\subsection{Guideline for choosing the kernel bandwidth \(\sigma\)}
 The kernel bandwidth controls the bias–variance tradeoff of the kernel density estimators \(\bigl(f_1(\cdot), f_2(\cdot)\bigr)\).
 It affects how accurately the smoothed densities approximate the worst-case distributions \(\PP_1^{\star}\) and \(\PP_1^{\star}\), respectively.
 In other words, the \emph{kernel bandwidth is directly linked to the generalization property of the detector beyond the training data to unseen residual samples}.
 For a practical guideline:
    \begin{itemize}[leftmargin=*, label = \(\circ\)]
        \item excessively small bandwidths may overfit the training data;

        \item excessively large bandwidths may blur the distinction between the nominal and attacked distributions.
    \end{itemize}
    \vspace{2mm}
    In our experiments, the bandwidth \emph{was chosen to balance the on-sample and off-sample detection performance}.
    Moreover, we fixed it across all experiments to facilitate consistent comparisons across different datasets, considering both Gaussian and non-Gaussian noise settings, as detailed in \S\ref{sec:num_exp} of the revised manuscript. However, one may also adopt more principled approaches, such as \cite{ref:SS-MCJ-91,ref:MPW-MCJ-95}, for selecting the bandwidth. 
    
    \vspace{2mm}
    
    We conducted additional experiments to analyze the effect of \(\sigma\) on the detector's performance.
    For these experiment, the data provided in Table \ref{tab:kernel_bandwidth_nongauss} below,  we fixed all parameter as in the manuscript, and conducted \(150\) Monte-Carlo simulations to compute the \(\ADD\) and \(\FAR\).
    We observed that for each threshold \(h\), increasing \(\sigma\) leads to increase in \(\ADD\) and consequently \(\FAR\) decreases, verifying that the detection performance is sensitive to \(\sigma\).
    This also means that as \(\sigma\) increased, the detection algorithm fails to distinguish between \(f_1(\cdot)\) and \(f_2(\cdot)\). Note that a moderate value of \(\sigma\) (which is \(0.5\) in our case) results in better generalization. 
    On the other hand, smaller values of \(\sigma\) will simply yields a detector that is overfit to the training data.
    
    \vspace{2mm}
    
    Note that \(\ADD(h)=0\) and \(\FAR(h) = 1\) denotes the failure of the detector due to \(h\) being too small. On the other hand, \(\FAR(h) = 0\) indicates that the threshold is very high, which also explain the corresponding high values of \(\ADD(h)\). 
\begin{table}[htpb]
    \centering
    \begin{adjustbox}{max width=\linewidth}
    \begin{tblr}{
      colspec = {c|c|cc|cc|cc|cc},
      row{1,2} = {azure9},
      hline{1,3,Z} = {2pt},
      hline{2} = {1pt},
      column{1-10} = {c},
    }

    & &
    \SetCell[c=2]{c} $\sigma_{\text{a}} = 0.5$ & &
    \SetCell[c=2]{c} $\sigma_{\text{a}} = 2.5$ & &
    \SetCell[c=2]{c} $\lambda = 0.5$ & &
    \SetCell[c=2]{c} $\lambda = 1.5$ & \\

    $h$ & $\sigma$ &
    $\ADD(h)$ & $\FAR(h)$ &
    $\ADD(h)$ & $\FAR(h)$ &
    $\ADD(h)$ & $\FAR(h)$ &
    $\ADD(h)$ & $\FAR(h)$ \\

    $10$ & $0.1$ &
    0 & 1 & 0 & 1 & 0 & 1 & 0 & 1 \\

    \SetRow{bg=lightgray}
    & $0.5$ &
    46.41 & 0.06 & 2.25 & 0.1 & 1.22 & 0.13 & 19.74 & 0.08 \\

    & $1$ &
    742.5 & 0.01 & 10.49 & 0.01 & 3.21 & 0 & 536.34 & 0 \\

    \hline[1pt]

    \SetCell[r=3]{c,valign=m} $50$ & $0.1$ &
    0.07 & 0.97 & 0.15 & 0.88 & 0.16 & 0.87 & 0.04 & 0.99 \\

    \SetRow{bg=lightgray}
    & $0.5$ &
    279.23 & 0 & 7.91 & 0 & 3.15 & 0 & 87.17 & 0 \\

    & $1$ &
    750 & 0 & 48.34 & 0 & 14.41 & 0 & 686.32 & 0 \\

    \hline[1pt]

    \SetCell[r=3]{c,valign=m} $80$ & $0.1$ &
    0.82 & 0.87 & 0.58 & 0.6 & 0.43 & 0.63 & 0.59 & 0.85 \\

    \SetRow{bg=lightgray}
    & $0.5$ &
    521 & 0 & 11.01 & 0 & 5.26 & 0 & 142.06 & 0 \\

    & $1$ &
    750 & 0 & 75.13 & 0 & 21.85 & 0 & 718.16 & 0 \\

    \end{tblr}
    \end{adjustbox}

    \vspace{0.7mm}
    \caption{
    $\ADD$ and $\FAR$ for different values of the kernel bandwidth $\sigma$
    for both Gaussian attacks (columns 3-6) and non-Gaussian attacks (columns 7-10), corresponding to Figures \ref{fig:gaus_att} and  \ref{fig:G_exp_att}, respectively.
    }
    \label{tab:kernel_bandwidth_nongauss}
\end{table}

\subsection{Guideline for choosing the radii \((\eps_1, \eps_2)\)}
\label{subsec-appen:choosing radii}
Note that the radii \((\eps_1, \eps_2)\) must be selected such that the well-posedness condition in Assumption \ref{assum:blanket} \(\eps_1 + \eps_2 < \Wdist(Q_1, Q_2)\) remains valid to ensure that the two ambiguity sets remain disjoint, reflecting the desired level of robustness against distributional uncertainty. 
In particular, for this condition:
    \begin{itemize}
        \item larger radii increase robustness against distributional uncertainty and consequently, result in a more conservative design;

        \item while smaller radii provide a less conservative detector but decrease the robustness of the design, as well.
    \end{itemize}
    \vspace{2mm}
    To demonstrate the effect of varying \((\eps_1, \eps_2)\) on the robustness performance of \(\algoname\), we conducted additional experiments.
    %
    To that end, we varied the radii \(\eps_1, \eps_2\) in  \(\aset[]{0.001, 0.01, 0.1, 1}\) for \(h\) taking values in \(\aset[]{5, 10, 50}\), and ran \(150\) Monte-Carlo simulations.
    We kept all other parameters unchanged, as reported in \S\ref{sec:num_exp} of the revised manuscript.

    The heatmaps, Figures \ref{fig:sensitivity_radii} and \ref{fig:sensitivity_radii_NG}, summarize how \((\eps_1, \eps_2)\) influences the detector's performance.
    Two observations are worth noting here: First, we report that larger values of \(\eps_1\) and \(\eps_2\) violate Assumption \ref{assum:blanket}.
    Consequently, for these values, \(\algoname\)'s performance becomes conservative, and it fails to perform and detect adversaries successfully, as indicated by extremely high \(\ADD\) and low \(\FAR\).
    And finally, for relatively smaller values of \(\eps_1\) and \(\eps_2\) satisfying Assumption \ref{assum:blanket}, \(\algoname\) detects adversaries (both in the Gaussian regime and in the non-Gaussian regime) and is in fact robust to variations in \(\eps_1\) and \(\eps_2\) to a large extent, as evident from figures \ref{fig:sensitivity_radii} and \ref{fig:sensitivity_radii_NG}.        
    \begin{figure}[!htpb]
    \centering

    \begin{subfigure}{\columnwidth}
        \centering
        \includegraphics[scale = 0.22]{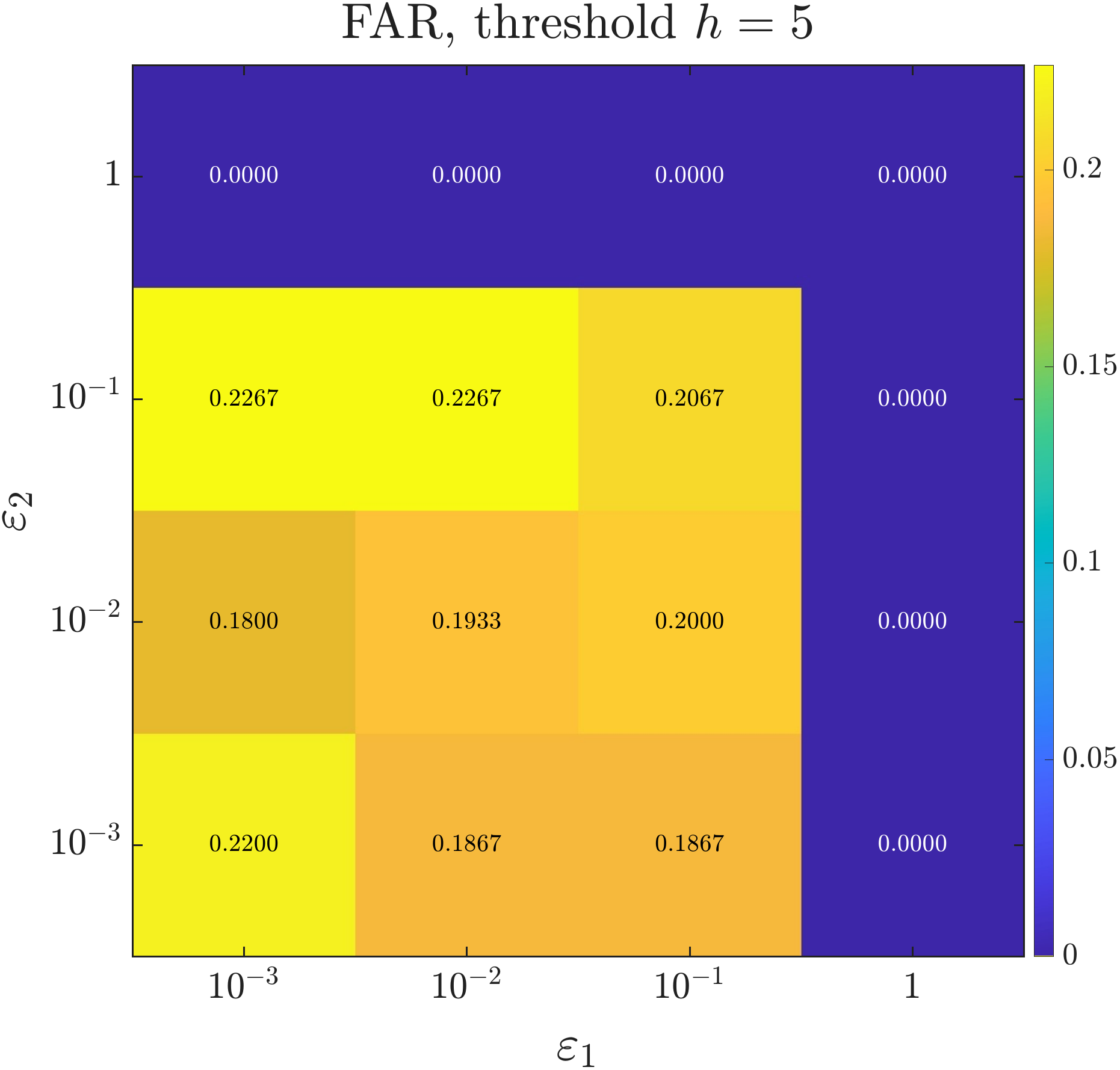}
        \includegraphics[scale = 0.22]{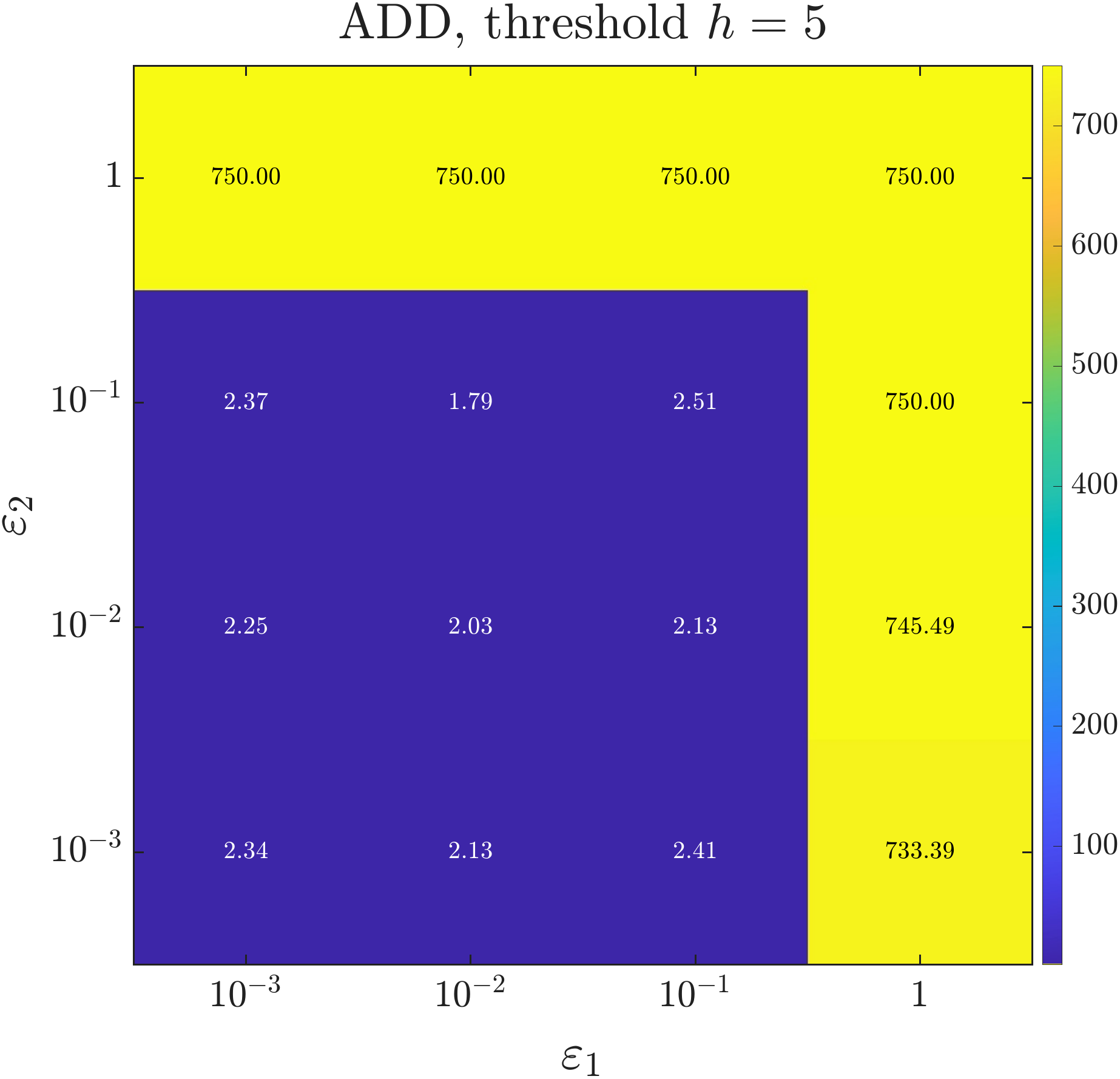}
        \caption{}
        \label{subfig:h_5}
    \end{subfigure}

    \vspace{2mm}

    \begin{subfigure}{\columnwidth}
        \centering
        \includegraphics[scale = 0.22]{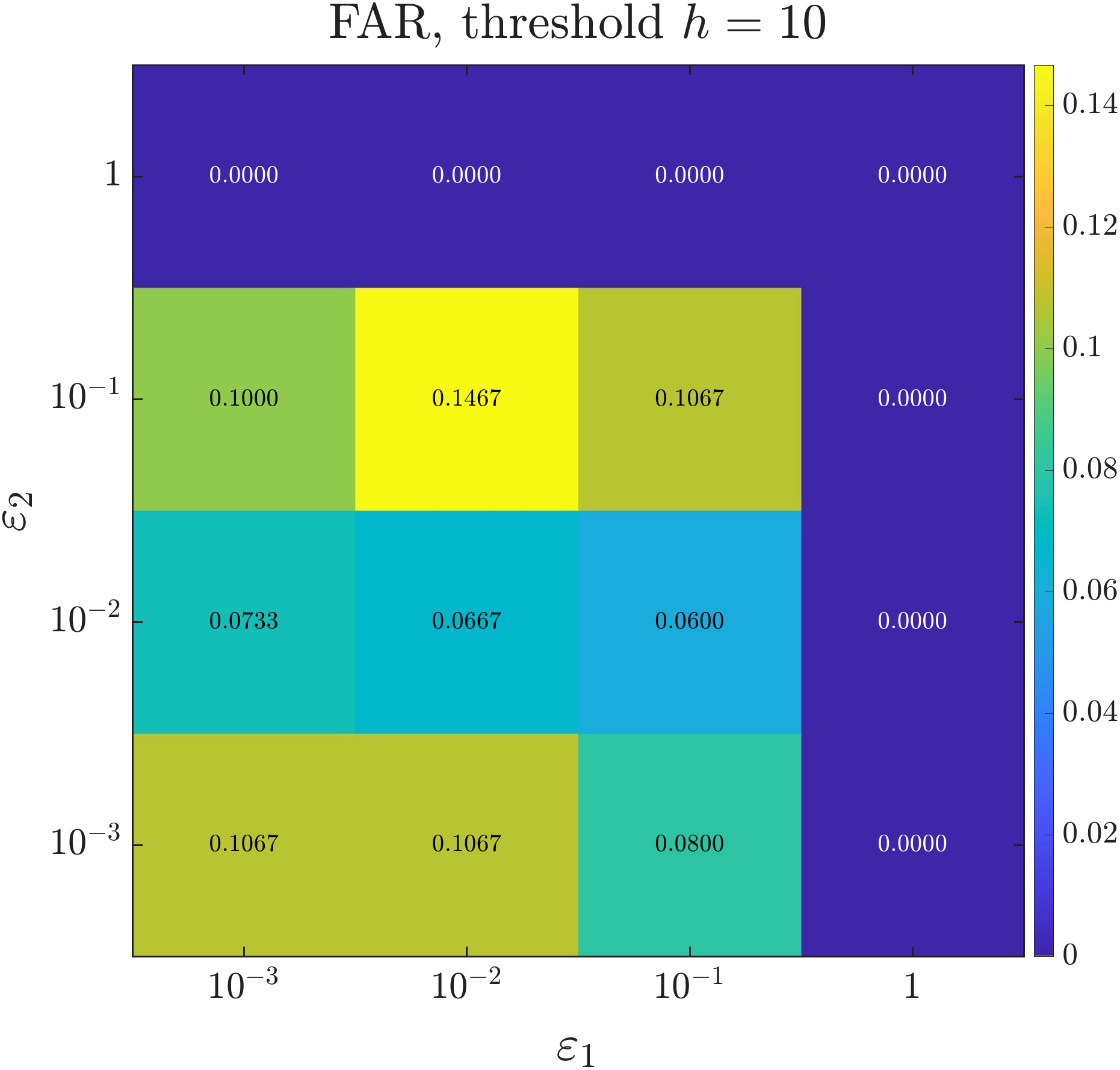}
        \includegraphics[scale = 0.22]{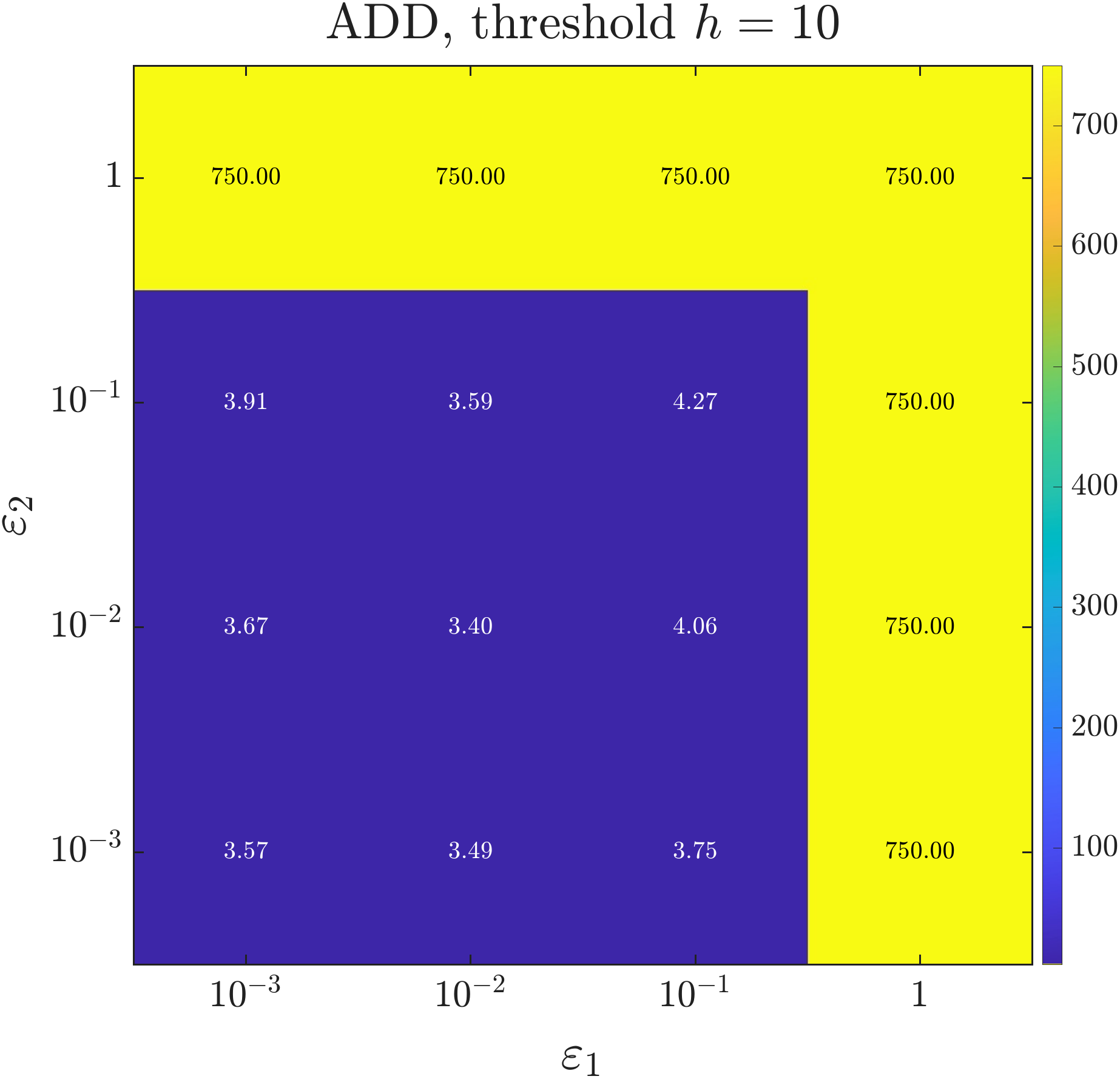}
        \caption{}
        \label{subfig:h_10}
    \end{subfigure}

    \vspace{2mm}

    \begin{subfigure}{\columnwidth}
        \centering
        \includegraphics[scale = 0.22]{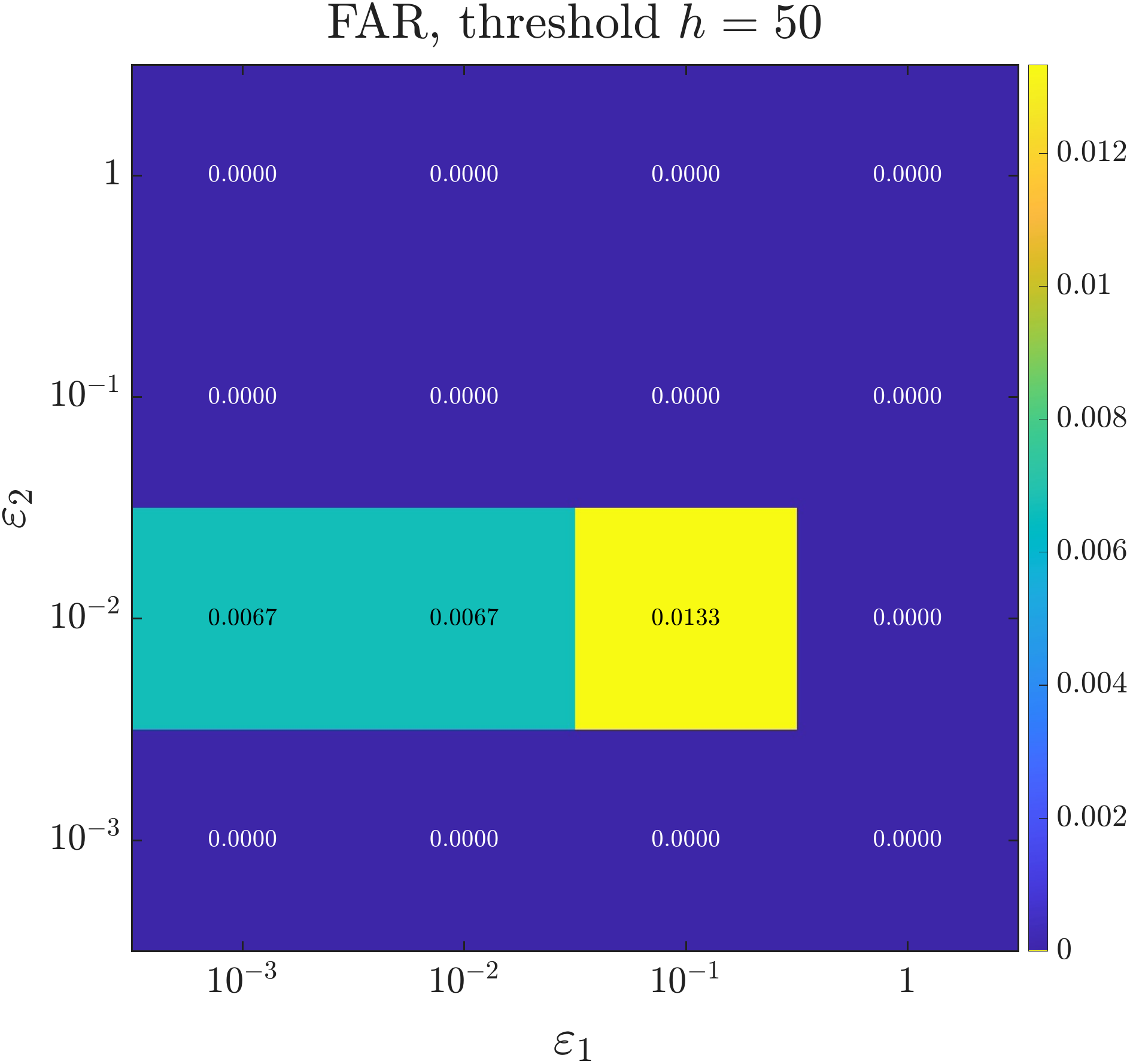}
        \includegraphics[scale = 0.22]{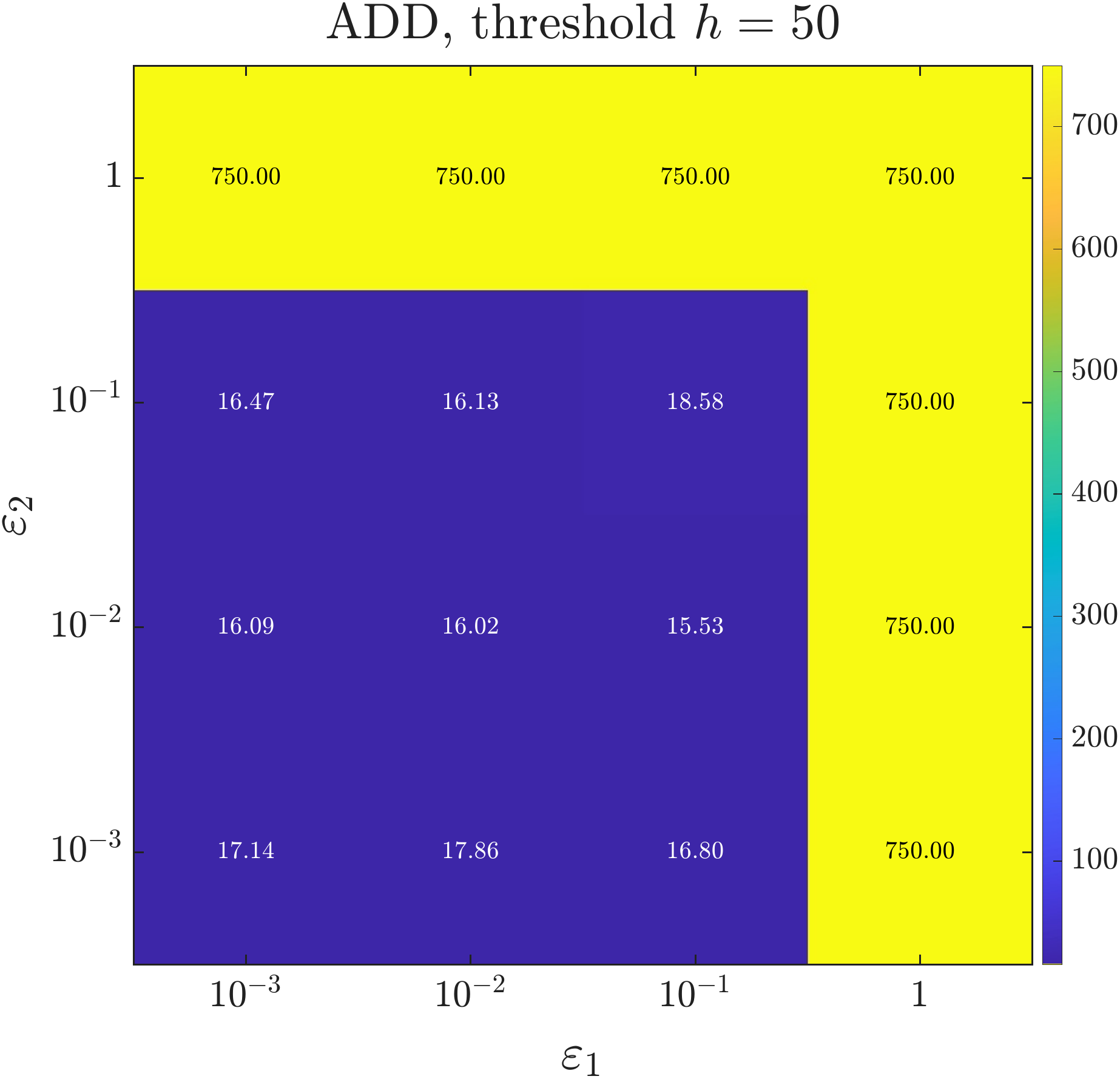}
        \caption{}
        \label{subfig:h_50}
    \end{subfigure}

    \caption{Heatmaps of \(\FAR(h)\) and \(\ADD(h)\) for \(h \in \aset[]{5,10,50}\), plotted against different values of the radii \(\eps_1\) and \(\eps_2\) under Gaussian attacks (\(\sigma_{\text{a}} = 1.5\)).}
    \label{fig:sensitivity_radii}
\end{figure}

    \begin{figure}[!htpb]
    \centering

    \begin{subfigure}{\columnwidth}
        \centering
        \includegraphics[scale = 0.22]{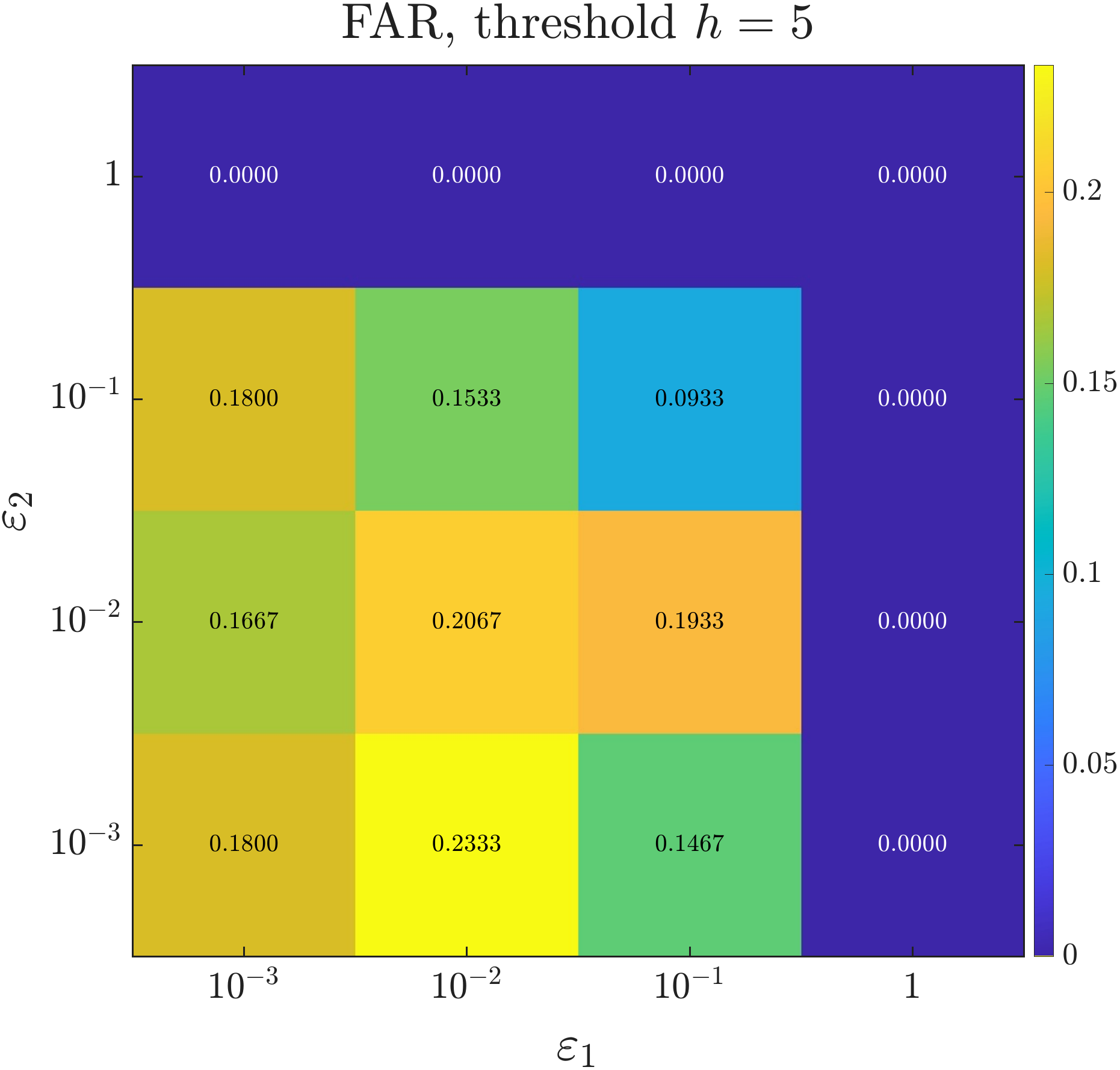}
        \includegraphics[scale = 0.22]{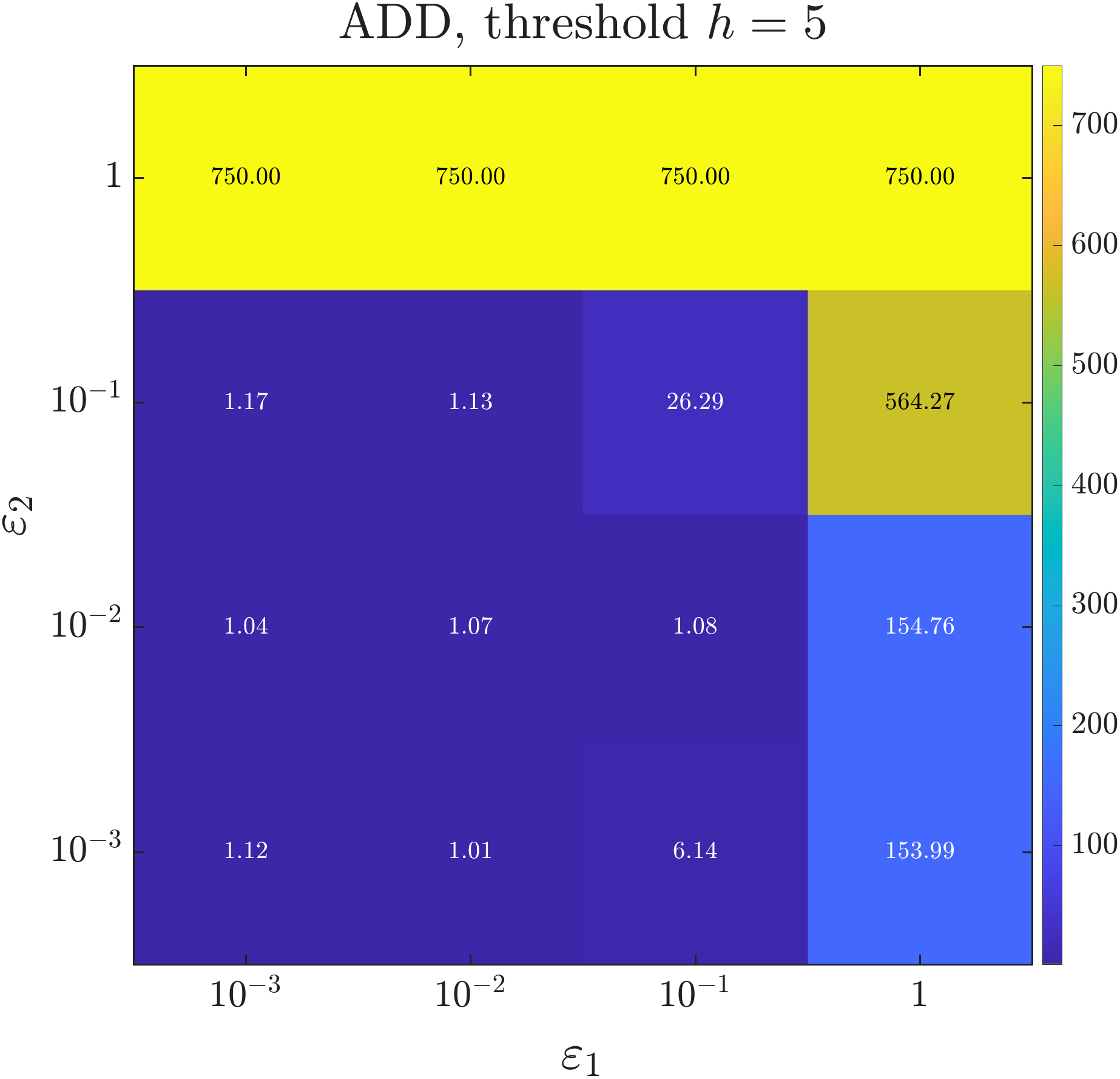}
        \caption{}
        \label{subfig:h_5_}
    \end{subfigure}

    \vspace{2mm}

    \begin{subfigure}{\columnwidth}
        \centering
        \includegraphics[scale = 0.22]{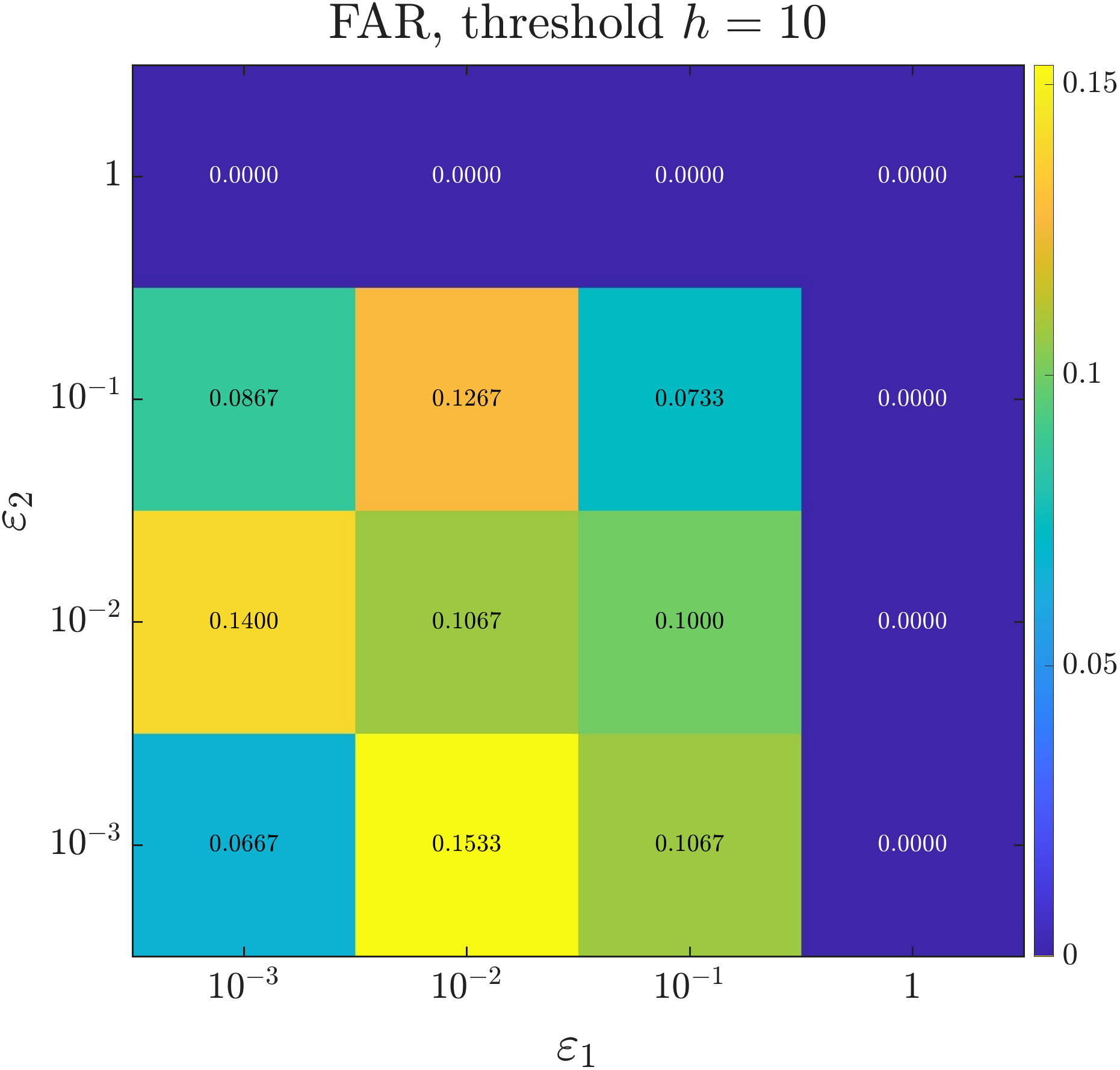}
        \includegraphics[scale = 0.22]{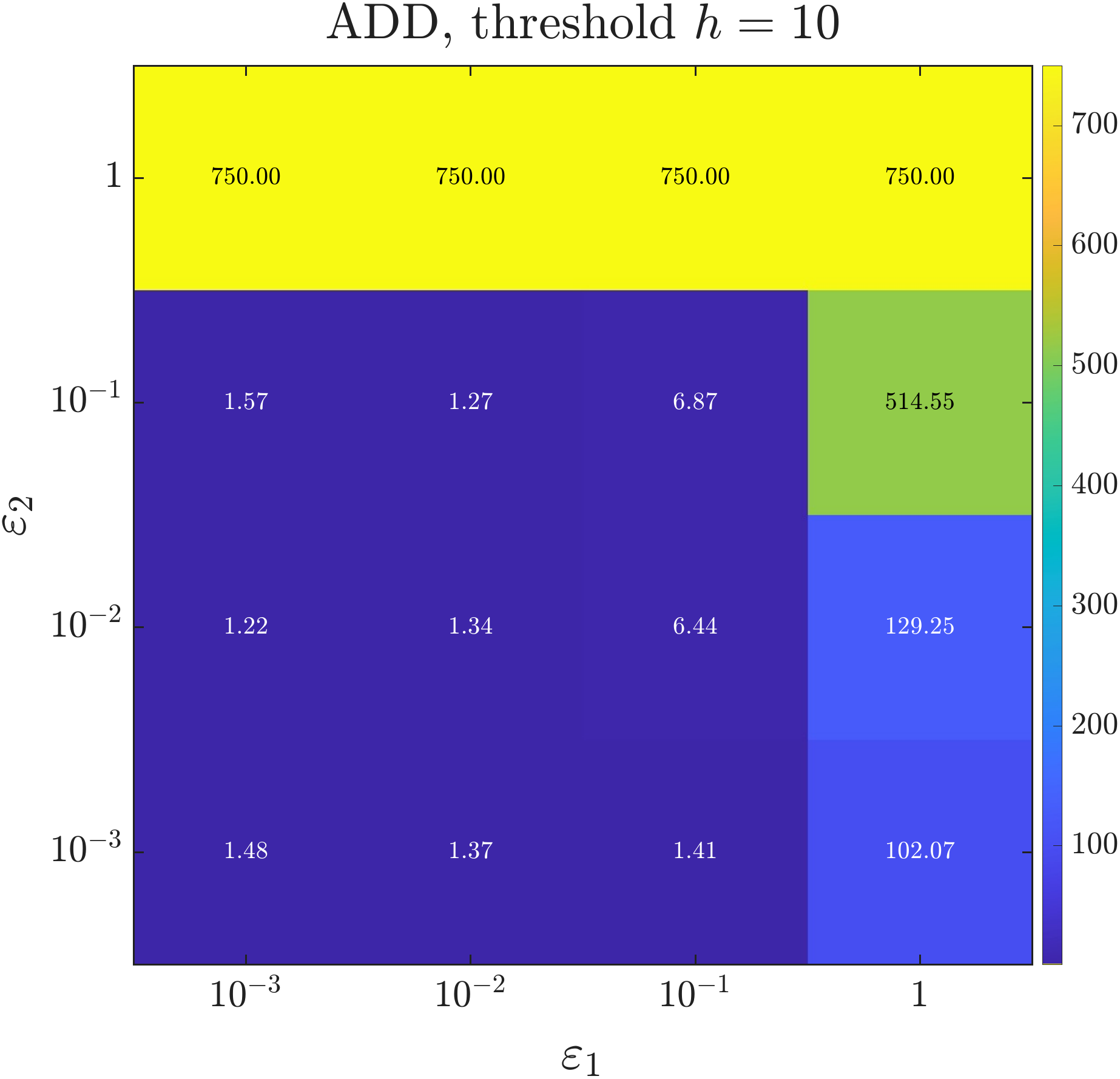}
        \caption{}
        \label{subfig:h_10_}
    \end{subfigure}

    \vspace{2mm}

    \begin{subfigure}{\columnwidth}
        \centering
        \includegraphics[scale = 0.22]{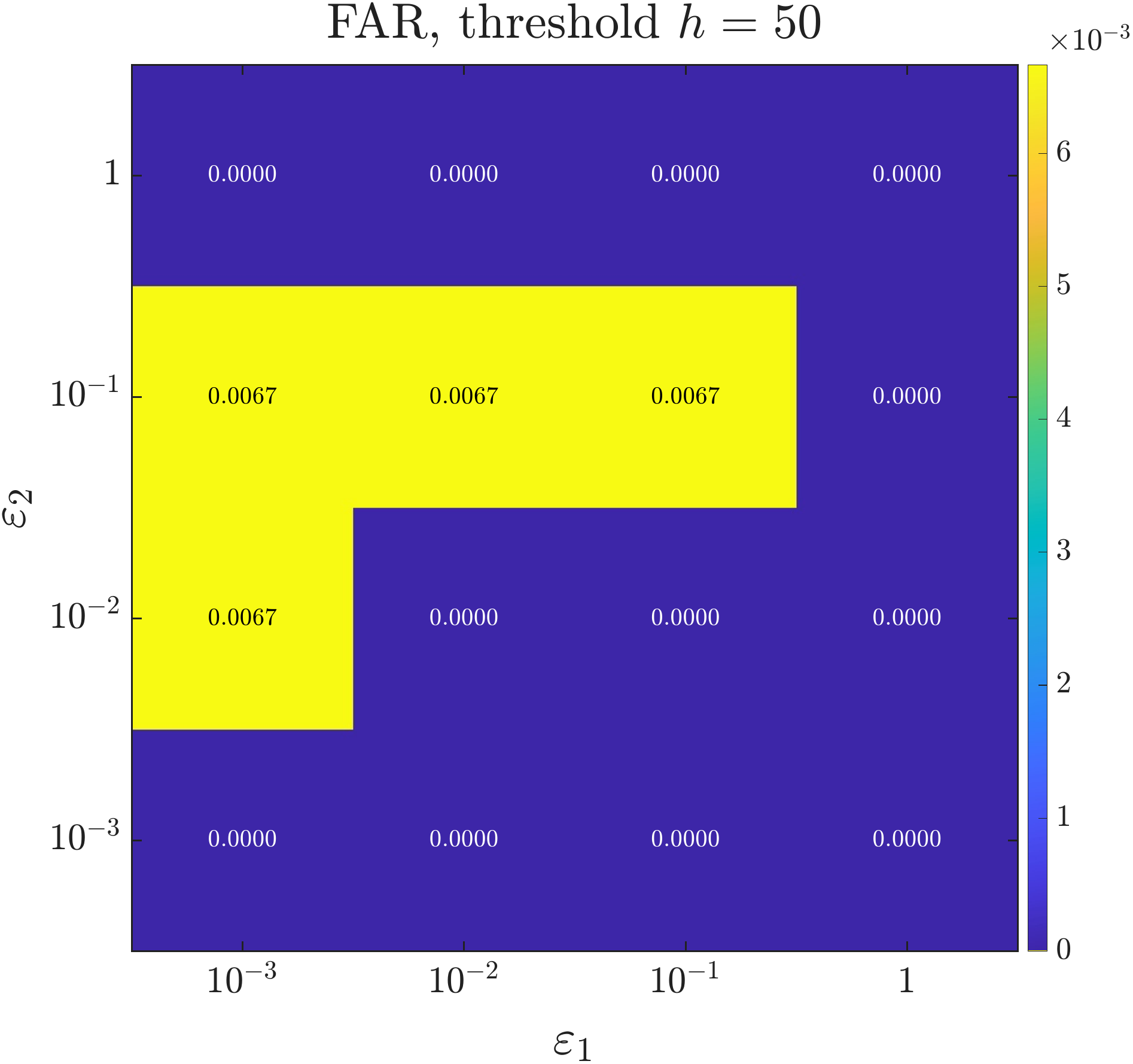}
        \includegraphics[scale = 0.22]{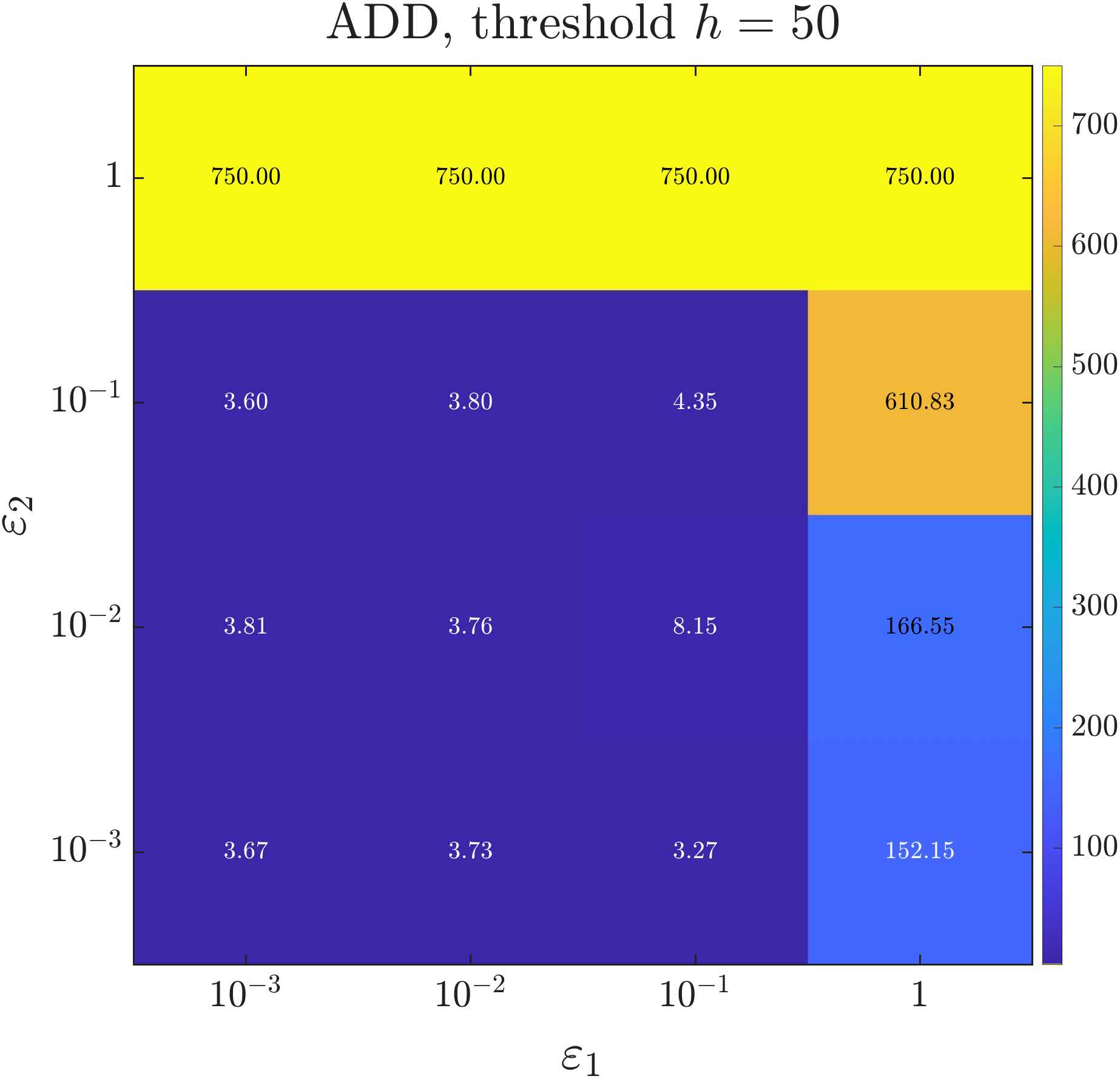}
        \caption{}
        \label{subfig:h_50_}
    \end{subfigure}

    \caption{Heatmaps of \(\FAR(h)\) and \(\ADD(h)\) for \(h \in \aset[]{5,10,50}\), plotted against different values of the radii \(\eps_1\) and \(\eps_2\) under non-Gaussian regime (for \(\sigma_{\text{a}}= 0.05\) and \(\lambda= 1.5\)).}
    \label{fig:sensitivity_radii_NG}
\end{figure}

\end{document}